\documentclass[leqno,12pt]{amsart}
\usepackage{amsmath,amstext,amssymb,amsopn,amsthm,mathrsfs}
\usepackage{epsfig,graphics,latexsym,graphicx,color}
\definecolor{darkgreen}{rgb}{0,.5,0}

\numberwithin{equation}{section}

\allowdisplaybreaks

\newtheorem{theorem}{Theorem}[section]

\newtheorem{corollary}[theorem]{Corollary}
\newtheorem{lemma}{Lemma}[section]

\begin{document}

\title[]{Sharp estimates for commutators of bilinear operators on Morrey type spaces}

\author[]{Dinghuai Wang, Jiang Zhou$^\ast$ and Zhidong Teng}
\address{College of Mathematics and System Sciences, Xinjiang University,  Urumqi 830046 \endgraf
         Republic of China}
\email{Wangdh1990@126.com; zhoujiang@xju.edu.cn; zhidong1960@163.com}
\thanks{The research was supported by National Natural Science Foundation
of China (Grant No.11661075 and No. 11271312). \\ \qquad * Corresponding author, zhoujiang@xju.edu.cn.}

\begin{abstract}
Denote by $T$ and $I_{\alpha}$ the bilinear Calder\'{o}n-Zygmund operators and bilinear fractional integrals, respectively. In this paper, it is proved that if $b_{1},b_{2}\in {\rm CMO}$ (the {\rm BMO}-closure of $C^{\infty}_{c}(\mathbb{R}^n)$), $[\Pi \vec{b},T]$ and $[\Pi\vec{b},I_{\alpha}]$ $(\vec{b}=(b_{1},b_{2}))$ are all the compact operators from $\mathcal{M}^{p_{0}}_{\vec{P}}$ (the norm of $\mathcal{M}^{p_{0}}_{\vec{P}}$ is strictly smaller than $2-$fold product of the Morrey norms) to $M^{q_{0}}_{q}$ for some suitable indexes $p_{0},p_{1},p_{2}$ and $q_{0},q$. Specially, we also show that if $b_{1}=b_{2}$, then $b_{1}, b_{2}\in {\rm CMO}$ is necessary for the compactness of $[\Pi\vec{b},I_{\alpha}]$ on Morrey space.
\end{abstract}
\maketitle

\maketitle

\section{Introduction}

The aim of the present paper is first: to obtain the boundedness and compactness of iterated commutators of bilinear operators acting on multi-Morrey spaces (a multi-Morrey norm is strictly smaller than $m-$fold product of the Morrey norms); and second: to characterize the compactness of the iterated commutators of bilinear fractional integral operators on Morrey spaces.

A well known result of Coifman, Rochberg and Weiss \cite{CRW} states that the commutator
$$[b,T](f)=bT(f)-T(bf)$$
is bounded on some $L^p$, $1<p<\infty$, if and only if $b\in \mathrm{BMO}$, where $T$ be the classical Calder\'{o}n-Zygmund operator. In 1978, Uchiyama \cite{U} refined the boundednss results on the commutator to compactness. This is a achieved by requiring the commutator with symbol to be in ${\rm CMO}$, which is the closure in ${\rm BMO}$ of the space of $C^{\infty}$ functions with compact support. In recent years, the compactness of commutators has been extensively studied already, Wang \cite{W} showed that the compactness of commutator of fractional integral operator and Ding et al. \cite{CD1}, \cite{CD2}, \cite{CD3}, \cite{CDW1}, \cite{CDW2} \cite{CDW3} also considered the compactness of commutators for some operators, such as the Riesz potential, singular integral, Marcinkiewicz integral in Morrey spaces. The interest in the compactness of commutators in complex analysis is from the connection between the commutators and the Hankel-type operators. In fact, the authors of \cite{KL1} and \cite{KL2} have applied commutator theory to give a compactness characterization of Hankel operators on holomorphic Hardy spaces $H^{2}(D)$, where $D$ is a bounded, strictly pseudoconvex domain in $\mathbb{C}^n$. It is perhaps for this important reason that the compactness of commutators attracted one¡¯s attention among researchers in PDEs.

Recently, many authors are interested in the multilinear setting, see \cite{BO}, \cite{G}, \cite{GK}, \cite{GT} and \cite{M}. The multilinear Calder\'{o}n-Zygmund theory originated in the works of Coifman and Meyer in the 70s, see e.g.\cite{CM1}, \cite{CM2}. Later on the topic was retaken by several authors; including Christ and Journ\'{e} \cite{CJ}, Kenig and Stein \cite{KS} and Grafakos and Torres \cite{GT}. The boundedness results for commutators with symbols in {\rm BMO} started to receive attention only a few years ago, see \cite{LOPTT}, \cite{P}, \cite{PT} or \cite{T}. Compactness results in the multilinear setting have just began to be studied. B\'{e}nyi et al. \cite{B1}, \cite{B2} and \cite{BT} showed that symbols in {\rm CMO} again produce compact commutators. Ding and Mei \cite{DM} consider the compactness of linear commutator of bilinear operators from product of Morrey spaces to Morrey spaces. In this paper, some sharp estimates for compactness of commutators of bilinear operators will be given; that is, it is proved that bilinear operators are all compact operators from multi-Morrey spaces(precise definition is given in the next secion) to Morrey spaces.

Another subject of this paper is to consider the characterization of compactness of the iterated commutator of bilinear fractional integral operators. For linear fractional integrals, the characterization of boundedness of the commutator was obtained by Chanillo \cite{C}, while the one for compactness is credited in \cite{CDW1} and \cite{W}. In the bilinear setting, in 2015, Chaffee and Torres \cite{CT} characterized the compactness of the linear commutators of bilinear fractional integral operators acting on product of Lebesgue spaces. In \cite{WZC}, we obtain the characterization of compactness of iterated commutators of bilinear fractional integral operators acting on product of Lebesgue spaces. In this paper, we will show that {\rm CMO} in fact characterizes compactness on Morrey spaces.
\vspace{0.5cm}
\section{Preliminaries and Main results}

\subsection{Bilinear Calder\'{o}n-Zygmund operator and its commutator}


Recall that bilinear singular integral operator $T$ is a bounded operator which satisfies
$$\|T(f_{1},f_{2})\|_{L^{p}}\leq C\|f_{1}\|_{L^{p_{1}}}\|f_{2}\|_{L^{p_{2}}},$$
for some $1<p_{1},p_{2}<\infty$ with $1/p=1/p_{1}+1/p_{2}$ and the function $K$, defined off the diagonal $y_{0}=y_{1}=y_{2}$ in $(\mathbb{R}^{n})^{2+1}$, satisfies the conditions as follow:

(1) The function $K$ satisfies the size condition.
$$|K(x,y_{1},y_{2})|
\leq\frac{C}{\big(|x-y_{1}|+|x-y_{2}|+|y_{1}-y_{2}|\big)^{2n}};$$

(2) The function $K$ satisfies the regularity condition. For some $\gamma>0$, if $|x-x'|\leq \frac{1}{2}\max\{|x-y_{1}|,|x-y_{2}|,|x-y_{2}|\}$
$$|K(x,y_{1},y_{2})-K(x',y_{1},y_{2})|
\leq\frac{C|x-x'|^{\gamma}}{\big(|x-y_{1}|+|x-y_{2}|+|y_{1}-y_{2}|\big)^{2n+\gamma}};$$
if $|y_{1}-y'_{1}|\leq \frac{1}{2}\max\{|x-y_{1}|,|x-y_{2}|,|y_{1}-y_{2}|\}$
$$|K(x,y_{1},y_{2})-K(x',y_{1},y_{2})|
\leq\frac{C|y_{1}-y'_{1}|^{\gamma}}{\big(|x-y_{1}|+|x-y_{2}|+|y_{1}-y_{2}|\big)^{2n+\gamma}};$$
if $|y_{2}-y'_{2}|\leq \frac{1}{2}\max\{|x-y_{1}|,|x-y_{2}|,|y_{1}-y_{2}|\}$
$$|K(x,y_{1},y_{2})-K(x',y_{1},y_{2})|
\leq\frac{C|y_{2}-y'_{2}|^{\gamma}}{\big(|x-y_{1}|+|x-y_{2}|+|y_{1}-y_{2}|\big)^{2n+\gamma}};$$

(3) If $x\notin {\rm supp} f_{1}\bigcap {\rm supp} f_{2}$, then
$$T(f_{1},f_{2})(x)=\int_{\mathbb{R}^n}\int_{\mathbb{R}^n}K(x,y_{1},y_{2})f_{1}(y_{1}) f_{2}(y_{2})dy_{1}dy_{2}.$$

It was shown that in \cite{GT} that if $\frac{1}{r_{1}}+\frac{1}{r_{2}}=\frac{1}{r}$, then an bilinear Calder\'{o}n-Zygmund operator satisies
$$T: L^{r_{1}}\times L^{r_{2}}\rightarrow L^{r}$$
when $1<r_{1},r_{2}<\infty$ and
$$T: L^{r_{1}}\times L^{r_{2}}\rightarrow L^{r,\infty}$$
when $1\leq r_{1},r_{2}<\infty$. In particular
$$T: L^{1}\times L^{1}\rightarrow L^{1/2,\infty}.$$

In 2003, P\'{e}rez and Torres in \cite{PT} defined the commutator $[\Pi\vec{b},T]$ as follows
\begin{eqnarray*}
&&[\Pi\vec{b},T](f_{1},f_{2})(x):=[b_{2},[b_{1},T]_{1}]_{2}(f_{1},f_{2})(x)\\
&&:=\int_{\mathbb{R}^n}\int_{\mathbb{R}^n}
(b_{1}(x)-b_{1}(y_{1}))(b_{2}(x)-b_{2}(y_{2}))K(x,y_{1},y_{2})f_{1}(y_{1})f_{2}(y_{2})dy_{1}dy_{2}.
\end{eqnarray*}
They also proved that if $b_{1},b_{2}\in {\rm BMO}$, then
$$[\Pi\vec{b},T]: L^{r_{1}}\times L^{r_{2}}\rightarrow L^{r}$$
when $1<r_{1},r_{2}<\infty$ with $\frac{1}{r}=\frac{1}{r_{1}}+\frac{1}{r_{2}}$.

The maximal operator $T_{*}$ of bilinear Calderon-Zygmund operator $T$ is defined by
$$
T_{*}(f_{1},f_{2})(x)=\sup_{\delta>0}\bigg|\int_{|x-y_{1}|+|x-y_{2}|>\delta}K(x,y_{1},y_{2})f_{1}(y_{1}) f_{2}(y_{2})dy_{1}dy_{2}\bigg|.
$$
In 2002, Grafakos and Torres in \cite{GT} proved that
$$T_{*}:L^{r_{1}}\times L^{r_{2}}\rightarrow L^{r}$$
when $1<r_{1},r_{2}<\infty$ with $\frac{1}{r}=\frac{1}{r_{1}}+\frac{1}{r_{2}}$.
\vspace{0.3cm}

\subsection{Bilinear fractional integral operator and its commutator}

It is well known that the fractional integral $\mathcal{I}_{\alpha}$ of order $\alpha (0<\alpha<n)$ plays an
important role in harmonic analysis, PDE and potential theory (see \cite{S}). Recall that
$\mathcal{I}_{\alpha}$ is defined by
$$\mathcal{I}_{\alpha}f(x)=\int_{\mathbb{R}^n}\frac{f(y)}{|x-y|^{n-\alpha}}dy.$$
For the bilinear case, the bilinear fractional integral operator $I_{\alpha}$, $0<\alpha<2n$, is defined by
$$I_{\alpha}(f_{1},f_{2})(x)=\int_{\mathbb{R}^n}\int_{\mathbb{R}^n}\frac{f_{1}(y_{1})f_{2}(y_{2})}{\big(|x-y_{1}|+|x-y_{2}|\big)^{2n-\alpha}}dy_{1}dy_{2}.$$
In this paper, we will consider the following equivalent operator
$$I_{\alpha}(f_{1},f_{2})(x)=\int_{\mathbb{R}^n}\int_{\mathbb{R}^n}\frac{f_{1}(y_{1})f_{2}(y_{2})}{\big(|x-y_{1}|^{2}+|x-y_{2}|^{2}\big)^{n-\alpha/2}}dy_{1}dy_{2}.$$
Its iterated commutator with $\vec{b}=(b_{1},b_{2})$ is given by
\begin{eqnarray*}
&&[\Pi\vec{b},I_{\alpha}](f_{1},f_{2})(x):=[b_{2},[b_{1},I_{\alpha}]_{1}]_{2}(f_{1},f_{2})(x)\\
&&=\int_{\mathbb{R}^n}\int_{\mathbb{R}^n}\frac{( b_{1}(x)-b_{1}(y_{1}))( b_{2}(x)-b_{2}(y_{2}))f_{1}(y_{1})f_{2}(y_{2})}{\big(|x-y_{1}|^{2}+|x-y_{2}|^{2}\big)^{n-\alpha/2}}dy_{1}dy_{2}.
\end{eqnarray*}
\vspace{0.3cm}

\subsection{Morrey type spaces}

The Morrey space was defined by Morrey \cite{M} in 1938, which is connected to certain problems in elliptic PDE. Later, the Morrey space was found to have many important applications to the Navier-Stokes equations \cite{K}, the Schr\"{o}dinger equations \cite{RV} and the potential analysis \cite{A1} and \cite{A2}.

Let $1\leq p\leq p_{0}<\infty$. The Morrey space $M^{p_{0}}_{p}$ is defined by the norm
$$\|f\|_{M^{p_{0}}_{p}}:=\sup_{Q}|Q|^{1/p_{0}}\Big(\frac{1}{|Q|}\int_{Q}|f(x)|^{p}dx\Big)^{1/p}<\infty.$$

In 2012, Iida et al. \cite{ISST} introduced the multi-Morrey norm as follow
$$\|(f_{1},f_{2})\|_{\mathcal{M}^{p_{0}}_{\vec{P}}}:
=\sup_{Q}|Q|^{1/p_{0}}\prod_{i=1}^{2}\Big(\frac{1}{|Q|}\int_{Q}|f_{i}(x)|^{p_{i}}dx\Big)^{1/p_{i}}<\infty.$$
They showed that Multi-Morrey norm is strictly smaller than $2-$fold product of the Morrey norms. They also proved that
$$T: \mathcal{M}^{p_{0}}_{\vec{P}}\rightarrow M^{p_{0}}_{p}~\text{and}~I_{\alpha}: \mathcal{M}^{p_{0}}_{\vec{P}}\rightarrow M^{q_{0}}_{q}$$
for some suitable indexes $p_{0},p_{1},p_{2}$ and $q_{0},q$. In this paper, we will consider the boundedness and compactness of the commutators $[\Pi\vec{b},T]$ and $[\Pi\vec{b},I_{\alpha}]$.

\subsection{Main results}

\vspace{0.3cm}

Now we return to our main results.

\begin{theorem}\label{main1}
Let $\vec{P}=(p_{1},p_{2})$, $1<p_{1},p_{2}<\infty$, $0< p\leq p_{0}<\infty$ with $\frac{1}{p}=\frac{1}{p_{1}}+\frac{1}{p_{2}}$.
Suppose that $T$ be a bilinear Calderon-Zygmund operator and $\vec{b}=(b_{1},b_{2})$ with $b_{1},b_{2}\in {\rm CMO}$, then $[\Pi\vec{b},T]$ is a compact operator from $\mathcal{M}^{p_{0}}_{\vec{P}}$ to $M^{p_{0}}_{q}$.
\end{theorem}

\begin{theorem}\label{main2}
Let $0<\alpha<2n, \vec{P}=(p_{1},p_{2}), 1<p_{1},p_{2}<\infty, 0<p\leq p_{0}<\infty, 0<q\leq q_{0}<\infty$ such that $\frac{1}{p}=\frac{1}{p_{1}}+\frac{1}{p_{2}}$,
$\frac{1}{q_{0}}=\frac{1}{p_{0}}-\frac{\alpha}{n}$ and $\frac{1}{q}=\frac{1}{p}-\frac{\alpha}{n}$. For the local integral functions $b_{1},b_{2}$ and $\vec{b}=(b_{1},b_{2})$, we have
\begin{enumerate}
\item [\rm(1)] if $b_{1},b_{2}\in {\rm CMO}$, then $[\Pi\vec{b},I_{\alpha}]$ is a compact operator from $\mathcal{M}^{p_{0}}_{\vec{P}}$ to $M^{q_{0}}_{q}$.
\item [\rm(2)] if $b_{1}=b_{2}$ and $[\Pi\vec{b},I_{\alpha}]$ is a compact operator from $\mathcal{M}^{p_{0}}_{\vec{P}}$ to $M^{q_{0}}_{q}$, then $b_{1},b_{2}\in {\rm CMO}$.
\end{enumerate}
\end{theorem}

\section{Main lemmas}
To prove Theorem \ref{main1} and Theorem \ref{main2}, we need the following results.

\begin{lemma}\label{lem1}
Let $1\leq p<p_{0}<\infty$, $1<p_{1},p_{2}<\infty$ with $\frac{1}{p}=\frac{1}{p_{1}}+\frac{1}{p_{2}}$.
Suppose that $T$ be a bilinear Calderon-Zygmund operator and $\vec{b}=(b_{1},b_{2})$ with $b_{1},b_{2}\in {\rm BMO}$, then
$$\|[\Pi\vec{b},T]\|_{M^{p_{0}}_{p}}\lesssim \|b_{1}\|_{\rm BMO}\|b_{2}\|_{\rm BMO}\|(f_{1},f_{2})\|_{\mathcal{M}^{p_{0}}_{\vec{P}}}.$$
\end{lemma}
\begin{proof}
Without loss of generality, we may assume that $\|b_{1}\|_{\rm BMO}=\|b_{2}\|_{\rm BMO}=1$. Fixing $Q:=Q(x_{0},r)$, we split $f_{i}$ into $f^{0}_{i}+f^{\infty}_{i}$ with $f^{0}_{i}=f_{i}\chi_{2Q}$ and $f^{\infty}_{i}=f_{i}\chi_{(2Q)^{c}}$, $i=1,2$. Then we need to verify the following inequalities:
\begin{equation}\label{main1.1}
{\rm I_{1}}:=|Q|^{1/p_{0}-1/p}\bigg(\int_{Q}\Big|[\Pi\vec{b},T](f^{0}_{1},f^{0}_{2})(x)\Big|^{p}dx\bigg)^{1/p}
\lesssim \|(f_{1},f_{2})\|_{\mathcal{M}^{p_{0}}_{\vec{P}}};
\end{equation}
\begin{equation}\label{main1.2}
{\rm I_{2}}:=|Q|^{1/p_{0}-1/p}\bigg(\int_{Q}\Big|[\Pi\vec{b},T](f^{0}_{1},f^{\infty}_{2})(x)\Big|^{p}dx\bigg)^{1/p}
\lesssim \|(f_{1},f_{2})\|_{\mathcal{M}^{p_{0}}_{\vec{P}}};
\end{equation}
\begin{equation}\label{main1.3}
{\rm I_{3}}:=|Q|^{1/p_{0}-1/p}\bigg(\int_{Q}\Big|[\Pi\vec{b},T](f^{\infty}_{1},f^{0}_{2})(x)\Big|^{p}dx\bigg)^{1/p}
\lesssim \|(f_{1},f_{2})\|_{\mathcal{M}^{p_{0}}_{\vec{P}}};
\end{equation}
\begin{equation}\label{main1.4}
{\rm I_{4}}:=|Q|^{1/p_{0}-1/p}\bigg(\int_{Q}\Big|[\Pi\vec{b},T](f^{\infty}_{1},f^{\infty}_{2})(x)\Big|^{p}dx\bigg)^{1/p}
\lesssim \|(f_{1},f_{2})\|_{\mathcal{M}^{p_{0}}_{\vec{P}}};
\end{equation}

We analyze each term separately. First, we give the proof of Eq. (\ref{main1.1}). The boundedness of $[\Pi\vec{b},T]$ from $L^{p_{1}}\times L^{p_{2}}$ to $L^{p}$ gives
\begin{eqnarray*}
{\rm I_{1}}&\lesssim& |Q|^{1/p_{0}-1/p}\prod_{i=1}^{2}\bigg(\int_{\mathbb{R}^n}|f^{0}_{i}(x)|^{p_{i}}dx\bigg)^{1/p_{i}}\\
&=& |Q|^{1/p_{0}-1/p}\prod_{i=1}^{2}\bigg(\int_{2Q}|f_{i}(x)|^{p_{i}}dx\bigg)^{1/p_{i}}\\
&\lesssim& \|(f_{1},f_{2})\|_{\mathcal{M}^{p_{0}}_{\vec{P}}}.
\end{eqnarray*}

To estimate ${\rm I_{2}}$, the operator $[\Pi\vec{b}, T]$ can be devided into the following parts:
\begin{eqnarray*}
&&[\Pi\vec{b}, T](f^{0}_{1},f^{\infty}_{2})(x)\\
&&=(b_{1}(x)-b_{1,2Q})(b_{2}(x)-b_{2,2Q})T(f^{0}_{1},f^{\infty}_{2})(x)+(b_{2}(x)-b_{2,2Q})T\big((b_{1,2Q}-b_{1})f^{0}_{1},f^{\infty}_{2}\big)(x)\\
&&\qquad+(b_{1}(x)-b_{1,2Q})T\big(f^{0}_{1},(b_{2,2Q}-b_{2})f^{\infty}_{2}\big)(x)+T\big((b_{1,2Q}-b_{1})f^{0}_{1},(b_{2,2Q}-b_{2})f^{\infty}_{2}\big)(x)\\
&&=:{\rm I_{21}+I_{22}+I_{23}+I_{24}},
\end{eqnarray*}
where $b_{i,2Q}=\frac{1}{|2Q|}\int_{2Q}b_{i}(x)dx$ for $i=1,2$.

Now, we give the estimates for ${\rm I_{21}, I_{22}, I_{23}, I_{24}},$ respectively. By the definition of $T$, we have
\begin{eqnarray*}
|T(f^{0}_{1},f^{\infty}_{1})(x)|&\lesssim& \int_{2Q}\int_{\mathbb{R}^n\backslash 2Q}\frac{|f_{1}(y_{1})||f_{2}(y_{2})|}{(|x-y_{1}|+|x-y_{2}|)^{2n}}dy_{2}dy_{1}\\
&\lesssim&\int_{2Q}|f_{1}(y_{1})|dy_{1}\sum_{k=1}^{\infty}\frac{1}{|2^{k}Q|^{2}}\int_{2^{k+1}Q\backslash2^{k}Q}|f_{2}(y_{2})|dy_{2}\\
&\lesssim&\sum_{k=1}^{\infty}|2^{k+1}Q|^{-1/p_{0}}\|(f_{1},f_{2})\|_{\mathcal{M}^{p_{0}}_{\vec{P}}}\\
&\lesssim&|Q|^{-1/p_{0}}\|(f_{1},f_{2})\|_{\mathcal{M}^{p_{0}}_{\vec{P}}}.
\end{eqnarray*}
Similar estimate gives
\begin{eqnarray*}
&&|T((b_{1}-b_{1,2Q})f^{0}_{1},f^{\infty}_{2})(x)|\\
&&\lesssim \int_{2Q}\int_{\mathbb{R}^n\backslash 2Q}\frac{|b_{1}(y_{1})-b_{1,2Q}||f_{1}(y_{1})||f_{2}(y_{2})|}{(|x-y_{1}|+|x-y_{2}|)^{2n}}dy_{2}dy_{1}\\
&&\lesssim \int_{2Q}|b_{1}(y_{1})-b_{1,2Q}||f_{1}(y_{1})|dy_{1}\sum_{k=1}^{\infty}\frac{1}{|2^{k}Q|^{2}}\int_{2^{k+1}Q\backslash2^{k}Q}|f_{2}(y_{2})|dy_{2}\\
&&\lesssim \bigg(\frac{1}{|2Q|}\int_{2Q}|b_{1}(y_{1})-b_{1,2Q}|^{p'_{1}}dy_{1}\bigg)^{1/p'_{1}}\sum_{k=1}^{\infty}|2^{k+1}Q|^{-1/p_{0}}\|(f_{1},f_{2})\|_{\mathcal{M}^{p_{0}}_{\vec{P}}}\\
&&\lesssim|Q|^{-1/p_{0}}\|(f_{1},f_{2})\|_{\mathcal{M}^{p_{0}}_{\vec{P}}}.
\end{eqnarray*}

From the fact that for $b\in {\rm BMO}$,
$$|b_{2^{k}Q}-b_{Q}|\lesssim k\|b\|_{\rm BMO},$$
which implies that
\begin{eqnarray*}
&&|T(f^{0}_{1},(b_{2}-b_{2,2Q})f^{\infty}_{2})(x)|\\
&&\lesssim \int_{2Q}\int_{\mathbb{R}^n\backslash 2Q}\frac{|b_{2}(y_{2})-b_{2,2Q}||f_{1}(y_{1})||f_{2}(y_{2})|}{(|x-y_{1}|+|x-y_{2}|)^{2n}}dy_{2}dy_{1}\\
&&\lesssim \int_{2Q}|f_{1}(y_{1})|dy_{1}\sum_{k=1}^{\infty}\frac{1}{|2^{k}Q|^{2}}\int_{2^{k+1}Q\backslash2^{k}Q}|b_{2}(y_{2})-b_{2,2Q}||f_{2}(y_{2})|dy_{2}\\
&&\lesssim \int_{2Q}|f_{1}(y_{1})|dy_{1}\sum_{k=1}^{\infty}\frac{1}{|2^{k}Q|^{2}}\int_{2^{k+1}Q}\big(|b_{2}(y_{2})-b_{2,2^{k+1}Q}|+|b_{2,2^{k+1}Q}-b_{2,2Q}|\big)|f_{2}(y_{2})|dy_{2}\\
&&\lesssim \sum_{k=1}^{\infty}k|2^{k+1}Q|^{-1/p_{0}}\|(f_{1},f_{2})\|_{\mathcal{M}^{p_{0}}_{\vec{P}}}\\
&&\lesssim|Q|^{-1/p_{0}}\|(f_{1},f_{2})\|_{\mathcal{M}^{p_{0}}_{\vec{P}}}.
\end{eqnarray*}

Finally, it remains to prove
$$|T((b_{1}-b_{1,2Q})f^{0}_{1},(b_{2}-b_{2,2Q})f^{\infty}_{2})(x)|\lesssim |Q|^{-1/p_{0}}\|(f_{1},f_{2})\|_{\mathcal{M}^{p_{0}}_{\vec{P}}}.$$
Note that
\begin{eqnarray*}
&&|T((b_{1}-b_{1,2Q})f^{0}_{1},(b_{2}-b_{2,2Q})f^{\infty}_{2})(x)|\\
&&\lesssim \int_{2Q}\int_{\mathbb{R}^n\backslash 2Q}\frac{|b_{1}(y_{1})-b_{1,2Q}||b_{2}(y_{2})-b_{2,2Q}||f_{1}(y_{1})||f_{2}(y_{2})|}{(|x-y_{1}|+|x-y_{2}|)^{2n}}dy_{2}dy_{1}\\
&&\lesssim \int_{2Q}|b_{1}(y_{1})-b_{1,2Q}||f_{1}(y_{1})|dy_{1}\sum_{k=1}^{\infty}\frac{1}{|2^{k}Q|^{2}}\int_{2^{k+1}Q\backslash2^{k}Q}|b_{2}(y_{2})-b_{2,2Q}||f_{2}(y_{2})|dy_{2}\\
&&\lesssim \sum_{k=1}^{\infty}k|2^{k+1}Q|^{-1/p_{0}}\|(f_{1},f_{2})\|_{\mathcal{M}^{p_{0}}_{\vec{P}}}\\
&&\lesssim|Q|^{-1/p_{0}}\|(f_{1},f_{2})\|_{\mathcal{M}^{p_{0}}_{\vec{P}}}.
\end{eqnarray*}
Thus, Minkowski's inequality and H\"{o}lder's inequality give that
\begin{eqnarray*}
{\rm I_{2}}&=&|Q|^{1/p_{0}-1/p}\bigg(\int_{Q}|{\rm I_{21}+I_{22}+I_{23}+I_{24}}|^{p}\bigg)^{1/p}\\
&\lesssim&\|(f_{1},f_{2})\|_{\mathcal{M}^{p_{0}}_{\vec{P}}}.
\end{eqnarray*}
We complete the proof of (\ref{main1.2}).

With the same idea of estimate for $I_{2}$, we can obtain the similar result for $I_{3}$.

To prove (\ref{main1.4}), we need only to show the following four inequalities.
\begin{equation}\label{main1.5}
|T(f^{\infty}_{1},f^{\infty}_{2})(x)|\lesssim |Q|^{-1/p_{0}}\|(f_{1},f_{2})\|_{\mathcal{M}^{p_{0}}_{\vec{P}}}
\end{equation}
\begin{equation}\label{main1.6}
|T((b_{1}-b_{1,2Q})f^{\infty}_{1},f^{\infty}_{2})(x)|\lesssim |Q|^{-1/p_{0}}\|(f_{1},f_{2})\|_{\mathcal{M}^{p_{0}}_{\vec{P}}}
\end{equation}
\begin{equation}\label{main1.7}
|T(f^{\infty}_{1},(b_{2}-b_{2,2Q})f^{\infty}_{2})(x)|\lesssim |Q|^{-1/p_{0}}\|(f_{1},f_{2})\|_{\mathcal{M}^{p_{0}}_{\vec{P}}}
\end{equation}
\begin{equation}\label{main1.8}
|T((b_{1}-b_{1,2Q})f^{\infty}_{1},(b_{2}-b_{2,2Q})f^{\infty}_{2})(x)|\lesssim |Q|^{-1/p_{0}}\|(f_{1},f_{2})\|_{\mathcal{M}^{p_{0}}_{\vec{P}}}
\end{equation}

Because (\ref{main1.6}), (\ref{main1.7}), (\ref{main1.8}) are completely analogous to (\ref{main1.5}), with a small difference, we only estimate (\ref{main1.5}). Set $\Omega_{0}:=\{(y_{1},y_{2}):|x-y_{1}|+|x-y_{2}|\leq 2r\}$ and $\Omega_{k}:=\{(y_{1},y_{2}): 2^{k}r< |x-y_{1}|+|x-y_{2}|\leq2^{k+1}r\}$, $k=1,2,\cdots$. Then
\begin{eqnarray*}
|T(f^{\infty}_{1},f^{\infty}_{2})(x)|
&\lesssim&\int_{\mathbb{R}^n\backslash 2Q}\int_{\mathbb{R}^n\backslash 2Q}\frac{|f_{1}(y_{1})||f_{2}(y_{2})|}{(|x-y_{1}|+|x-y_{2}|)^{2n}}dy_{1}dy_{2}\\
&\lesssim&\iint_{\mathbb{R}^{2n}\backslash \Omega_{0}}\frac{|f_{1}(y_{1})||f_{2}(y_{2})|}{(|x-y_{1}|+|x-y_{2}|)^{2n}}dy_{1}dy_{2}\\
&\lesssim&\sum_{k=1}^{\infty}\iint_{\Omega_{k}}\frac{|f_{1}(y_{1})||f_{2}(y_{2})|}{(|x-y_{1}|+|x-y_{2}|)^{2n}}dy_{1}dy_{2}\\
&\lesssim&\sum_{k=1}^{\infty}\frac{1}{(2^{k}r)^{2n}}\prod_{i=1}^{2}\int_{2^{k+1}B}|f_{i}(y_{i})|dy_{i}\\
&\lesssim&\sum_{k=1}^{\infty}|2^{k}Q|^{-1/p_{0}}\|(f_{1},f_{2})\|_{\mathcal{M}^{p_{0}}_{\vec{P}}}\\
&\lesssim&|Q|^{-1/p_{0}}\|(f_{1},f_{2})\|_{\mathcal{M}^{p_{0}}_{\vec{P}}},
\end{eqnarray*}
where $B:=B(x,r)$.

Combining the estimates above, we have
$$\|[\Pi\vec{b},T](f_{1},f_{2})\|_{M^{p_{0}}_{p}}\lesssim \|(f_{1},f_{2})\|_{\mathcal{M}^{p_{0}}_{\vec{P}}}.$$
We complete the proof of Lemma \ref{lem1}.
\end{proof}

Now, we give the $M^{p_{0}}_{p}-$boundedness for a general bi-sublinear operator which satisfies some control conditions.
\begin{lemma}\label{lem2}
Let $S$ is a bi-sublinear operator satisfies
$$|S(f_{1},f_{2})(x)|\lesssim \int_{\mathbb{R}^n}\int_{\mathbb{R}^n}\frac{|f_{1}(y_{1})f_{2}(y_{2})|}{(|x-y_{1}|+|x-y_{2}|)^{2n}}dy_{1}dy_{2},$$
and for $1\leq p<\infty$, $1<p_{1},p_{2}<\infty$ with $\frac{1}{p}=\frac{1}{p_{1}}+\frac{1}{p_{2}}$, $S$ is bounded from $L^{p_{1}}\times L^{p_{2}}$ to $L^{p}$. Then for $1\leq p<p_{0}<\infty$ and $\vec{P}=(p_{1},p_{2})$, $S$ is bounded from $\mathcal{M}^{p_{0}}_{\vec{P}}$ to $M^{p_{0}}_{p}$; that is
$$\|S(f_{1},f_{2})\|_{M^{p_{0}}_{p}}\lesssim \|(f_{1},f_{2})\|_{\mathcal{M}^{p_{0}}_{\vec{P}}}.$$
\end{lemma}
\begin{proof}
Fixing $Q:=Q(x_{0},r)$ and we write
$$f_{i}=f_{i}\chi_{2Q}+\sum_{k=1}^{\infty}f_{i}\chi_{2^{k+1}Q\backslash2^{k}Q}=:\sum_{k=0}^{\infty}f^{k}_{i}.$$

Then the $L^{p}-$boundedness of $S$ yields
\begin{eqnarray*}
|Q|^{1/p_{0}-1/p}\bigg(\int_{Q}|S(f^{0}_{1},f^{0}_{2})(x)|^{p}dx\bigg)^{1/p}&\lesssim&|Q|^{1/p_{0}-1/p}\|f^{0}_{1}\|_{L^{p_{1}}}\|f^{0}_{2}\|_{L^{p_{2}}}\\
&\lesssim& \|(f_{1},f_{2})\|_{\mathcal{M}^{p_{0}}_{\vec{P}}}.
\end{eqnarray*}
To complete the proof of Lemma \ref{lem2}, it remains to show the following four inequalities.
\begin{equation}\label{3.2.1}
\sum_{k=1}^{\infty}\bigg(\int_{Q}|S(f^{0}_{1},f^{k}_{2})(x)|^{p}dx\bigg)^{1/p}\lesssim |Q|^{1/p-1/p_{0}}\|(f_{1},f_{2})\|_{\mathcal{M}^{p_{0}}_{\vec{P}}};
\end{equation}
\begin{equation}\label{3.2.2}
\sum_{j=1}^{\infty}\bigg(\int_{Q}|S(f^{j}_{1},f^{0}_{2})(x)|^{p}dx\bigg)^{1/p}\lesssim |Q|^{1/p-1/p_{0}}\|(f_{1},f_{2})\|_{\mathcal{M}^{p_{0}}_{\vec{P}}};
\end{equation}
\begin{equation}\label{3.2.3}
\sum_{j=1}^{\infty}\sum_{k=1}^{j}\bigg(\int_{Q}|S(f^{j}_{1},f^{k}_{2})(x)|^{p}dx\bigg)^{1/p}\lesssim |Q|^{1/p-1/p_{0}}\|(f_{1},f_{2})\|_{\mathcal{M}^{p_{0}}_{\vec{P}}};
\end{equation}
\begin{equation}\label{3.2.4}
\sum_{j=1}^{\infty}\sum_{k=j}^{\infty}\bigg(\int_{Q}|S(f^{j}_{1},f^{k}_{2})(x)|^{p}dx\bigg)^{1/p}\lesssim |Q|^{1/p-1/p_{0}}\|(f_{1},f_{2})\|_{\mathcal{M}^{p_{0}}_{\vec{P}}};
\end{equation}

By the symmetry, we need only to prove (\ref{3.2.1}) and (\ref{3.2.3}). First, we give the proof of (\ref{3.2.1}). Note that
\begin{eqnarray*}
|S(f^{0}_{1},f^{k}_{2})(x)|&\lesssim&\int_{2Q}\int_{2^{k+1}Q\backslash 2^{k}Q}\frac{|f_{1}(y_{1})||f_{2}(y_{2})|}{(|x-y_{1}|+|x-y_{2}|)^{2n}}dy_{1}dy_{2}\\
&\lesssim& |2^{k}Q|^{-2}\int_{2^{k+1}Q}\int_{2^{k+1}Q}|f_{1}(y_{1})||f_{2}(y_{2})|dy_{1}dy_{2}\\
&\lesssim&|2^{k}Q|^{-1/p_{0}}\|(f_{1},f_{2})\|_{\mathcal{M}^{p_{0}}_{\vec{P}}},
\end{eqnarray*}
which implies that
\begin{eqnarray*}
\sum_{k=1}^{\infty}\bigg(\int_{Q}|S(f^{0}_{1},f^{k}_{2})(x)|^{p}dx\bigg)^{1/p}
&\lesssim&|Q|^{1/p}\sum_{k=1}^{\infty}|2^{k}Q|^{-1/p_{0}}\|(f_{1},f_{2})\|_{\mathcal{M}^{p_{0}}_{\vec{P}}}\\
&\lesssim& |Q|^{1/p-1/p_{0}}\|(f_{1},f_{2})\|_{\mathcal{M}^{p_{0}}_{\vec{P}}}.
\end{eqnarray*}

For $j\geq k$, we also have
\begin{eqnarray*}
S(f^{j}_{1},f^{k}_{2})(x)&\lesssim&|2^{j}Q|^{-1/p_{0}}\|(f_{1},f_{2})\|_{\mathcal{M}^{p_{0}}_{\vec{P}}}.
\end{eqnarray*}
Thus,
\begin{eqnarray*}
\sum_{j=1}^{\infty}\sum_{k=1}^{j}\bigg(\int_{Q}|S(f^{j}_{1},f^{k}_{2})(x)|^{p}dx\bigg)^{1/p}
&\lesssim&|Q|^{1/p}\sum_{j=1}^{\infty}j|2^{j}Q|^{-1/p_{0}}\|(f_{1},f_{2})\|_{\mathcal{M}^{p_{0}}_{\vec{P}}}\\
&\lesssim& |Q|^{1/p-1/p_{0}}\|(f_{1},f_{2})\|_{\mathcal{M}^{p_{0}}_{\vec{P}}}.
\end{eqnarray*}
Thus, we complete the proof of the Lemma \ref{lem2}.
\end{proof}

Since bilinear maximal Calderon-Zygmund operator $T_{*}$ satisfies the condition as in Lemma \ref{lem2}, we get immediately the sharp bounds for $T_{*}$ on Morrey spaces.
\begin{corollary}\label{c}
Let $\vec{P}=(p_{1},p_{2}), 1\leq p<p_{0}<\infty$, $1<p_{1},p_{2}<\infty$ with $\frac{1}{p}=\frac{1}{p_{1}}+\frac{1}{p_{2}}$.
Suppose that $T_{*}$ be a bilinear maximal Calderon-Zygmund operator, then
$$\|T_{*}(f_{1},f_{2})\|_{M^{p_{0}}_{p}}\lesssim \|(f_{1},f_{2})\|_{\mathcal{M}^{p_{0}}_{\vec{P}}}.$$
\end{corollary}

\begin{lemma}\label{lem3}
Under the hypotheses of Theorem \ref{main2}. For the local integral functions $b_{1},b_{2}$ and $\vec{b}=(b_{1},b_{2})$, we have
\begin{enumerate}
\item [\rm(1)] if $b_{1},b_{2}\in {\rm BMO}$, then $[\Pi\vec{b},I_{\alpha}]$ is a bounded operator from $\mathcal{M}^{p_{0}}_{\vec{P}}$ to $M^{q_{0}}_{q}$.
\item [\rm(2)] if $b_{1}=b_{2}$ and $[\Pi\vec{b},I_{\alpha}]$ is a bounded operator from $\mathcal{M}^{p_{0}}_{\vec{P}}$ to $M^{q_{0}}_{q}$, then $b_{1},b_{2}\in {\rm BMO}$.
\end{enumerate}
\end{lemma}

\begin{proof}
Assume that $b_{1},b_{2}\in {\rm BMO}$. For any cube $Q$, we split $f_{i}$ into $f^{0}_{i}+f^{\infty}_{i}$ with $f^{0}_{i}=f_{i}\chi_{2Q}$ and $f^{\infty}_{i}=f_{i}\chi_{(2Q)^{c}}$, $i=1,2$. Then we need to verify the following inequalities:
\begin{equation}\label{3.3.1}
{\rm J_{1}}:=|Q|^{1/q_{0}-1/q}\bigg(\int_{Q}\Big|[\Pi\vec{b},I_{\alpha}](f^{0}_{1},f^{0}_{2})(x)\Big|^{q}dx\bigg)^{1/q}\lesssim \|(f_{1},f_{2})\|_{\mathcal{M}^{p_{0}}_{\vec{P}}};
\end{equation}
\begin{equation}\label{3.3.2}
{\rm J_{2}}:=|Q|^{1/q_{0}-1/q}\bigg(\int_{Q}\Big|[\Pi\vec{b},I_{\alpha}](f^{0}_{1},f^{\infty}_{2})(x)\Big|^{q_{0}}dx\bigg)^{1/q_{0}}\lesssim \|(f_{1},f_{2})\|_{\mathcal{M}^{p_{0}}_{\vec{P}}};
\end{equation}
\begin{equation}\label{3.3.3}
{\rm J_{3}}:=|Q|^{1/q_{0}-1/q}\bigg(\int_{Q}\Big|[\Pi\vec{b},I_{\alpha}](f^{\infty}_{1},f^{0}_{2})(x)\Big|^{q_{0}}dx\bigg)^{1/q_{0}}\lesssim \|(f_{1},f_{2})\|_{\mathcal{M}^{p_{0}}_{\vec{P}}};
\end{equation}
\begin{equation}\label{3.3.4}
{\rm J_{4}}:=|Q|^{1/q_{0}-1/q}\bigg(\int_{Q}\Big|[\Pi\vec{b},I_{\alpha}](f^{\infty}_{1},f^{\infty}_{2})(x)\Big|^{q_{0}}dx\bigg)^{1/q_{0}}\lesssim \|(f_{1},f_{2})\|_{\mathcal{M}^{p_{0}}_{\vec{P}}}.
\end{equation}

By the boundedness of $[\Pi\vec{b},I_{\alpha}]$ from $L^{p_{1}}\times L^{p_{2}}$ to $L^{q}$, we have
\begin{eqnarray*}
{\rm J_{1}}&\lesssim& |Q|^{1/q_{0}-1/q}\prod_{i=1}^{2}\bigg(\int_{\mathbb{R}^n}|f^{0}_{i}(x)|^{p_{i}}dx\bigg)^{1/p_{i}}\\
&\lesssim& |Q|^{1/p_{0}-1/p}\prod_{i=1}^{2}\bigg(\int_{2Q}|f_{i}(x)|^{p_{i}}dx\bigg)^{1/p_{i}}\\
&\lesssim& \|(f_{1},f_{2})\|_{\mathcal{M}^{p_{0}}_{\vec{P}}}.
\end{eqnarray*}

The terms ${\rm J_{2}, J_{3}, J_{4}}$ are estimates, with slight changes, using the same tools as in the proof for $[\Pi\vec{b},T]$. For example, if we consider the $J_{2}$ term, we first give the estimates for some operators. First,
\begin{eqnarray*}
|I_{\alpha}(f^{0}_{1},f^{\infty}_{1})(x)|&\lesssim& \int_{2Q}\int_{\mathbb{R}^n\backslash 2Q}\frac{|f_{1}(y_{1})||f_{2}(y_{2})|}{(|x-y_{1}|^{2}+|x-y_{2}|^{2})^{n-\alpha/2}}dy_{2}dy_{1}\\
&\lesssim&\int_{2Q}|f_{1}(y_{1})|dy_{1}\sum_{k=1}^{\infty}\frac{1}{|2^{k}Q|^{2-\alpha/n}}\int_{2^{k+1}Q\backslash2^{k}Q}|f_{2}(y_{2})|dy_{2}\\
&\lesssim&\sum_{k=1}^{\infty}|2^{k+1}Q|^{\alpha/n-1/p_{0}}\|(f_{1},f_{2})\|_{\mathcal{M}^{p_{0}}_{\vec{P}}}\\
&\lesssim&|Q|^{\alpha/n-1/p_{0}}\|(f_{1},f_{2})\|_{\mathcal{M}^{p_{0}}_{\vec{P}}}.
\end{eqnarray*}
Second,
\begin{eqnarray*}
&&|I_{\alpha}((b_{1}-b_{1,2Q})f^{0}_{1},f^{\infty}_{2})(x)|\\
&&\lesssim \int_{2Q}\int_{\mathbb{R}^n\backslash 2Q}\frac{|b_{1}(y_{1})-b_{1,2Q}||f_{1}(y_{1})||f_{2}(y_{2})|}{(|x-y_{1}|^{2}+|x-y_{2}|^{2})^{n-\alpha/2}}dy_{2}dy_{1}\\
&&\lesssim \int_{2Q}|b_{1}(y_{1})-b_{1,2Q}||f_{1}(y_{1})|dy_{1}\sum_{k=1}^{\infty}\frac{1}{|2^{k}Q|^{2-\alpha/n}}\int_{2^{k+1}Q\backslash2^{k}Q}|f_{2}(y_{2})|dy_{2}\\
&&\lesssim  \bigg(\frac{1}{|2Q|}\int_{2Q}|b_{1}(y_{1})-b_{1,2Q}|^{p'_{1}}dy_{1}\bigg)^{1/p'_{1}}\sum_{k=1}^{\infty}|2^{k+1}Q|^{\alpha/n-1/p_{0}}\|(f_{1},f_{2})\|_{\mathcal{M}^{p_{0}}_{\vec{P}}}\\
&&\lesssim \|b_{1}\|_{\rm BMO}|Q|^{\alpha/n-1/p_{0}}\|(f_{1},f_{2})\|_{\mathcal{M}^{p_{0}}_{\vec{P}}}.
\end{eqnarray*}
Third,
\begin{eqnarray*}
&&|I_{\alpha}(f^{0}_{1},(b_{2}-b_{2,2Q})f^{\infty}_{2})(x)|\\
&&\lesssim \int_{2Q}\int_{\mathbb{R}^n\backslash 2Q}\frac{|b_{2}(y_{2})-b_{2,2Q}||f_{1}(y_{1})||f_{2}(y_{2})|}{(|x-y_{1}|^{2}+|x-y_{2}|^{2})^{n-\alpha/2}}dy_{2}dy_{1}\\
&&\lesssim \int_{2Q}|f_{1}(y_{1})|dy_{1}\sum_{k=1}^{\infty}\frac{1}{|2^{k}Q|^{2-\alpha/n}}\int_{2^{k+1}Q\backslash2^{k}Q}|b_{2}(y_{2})-b_{2,2Q}||f_{2}(y_{2})|dy_{2}\\
&&\lesssim \sum_{k=1}^{\infty}|2^{k+1}Q|^{\alpha/n-1/p_{0}} \bigg(\frac{1}{|2^{k+1}Q|}\int_{2^{k+1}Q}|b_{2}(y_{2})-b_{2,2Q}|^{p'_{2}}dy_{2}\bigg)^{1/p'_{2}}\|(f_{1},f_{2})\|_{\mathcal{M}^{p_{0}}_{\vec{P}}}\\
&&\lesssim\|b_{2}\|_{\rm BMO}\sum_{k=1}^{\infty}k|2^{k+1}Q|^{\alpha/n-1/p_{0}}\|(f_{1},f_{2})\|_{\mathcal{M}^{p_{0}}_{\vec{P}}}\\
&&\lesssim \|b_{2}\|_{\rm BMO}|Q|^{\alpha/n-1/p_{0}}\|(f_{1},f_{2})\|_{\mathcal{M}^{p_{0}}_{\vec{P}}}.
\end{eqnarray*}

Finally,
\begin{eqnarray*}
&&|I_{\alpha}((b_{1}-b_{1,2Q})f^{0}_{1},(b_{2}-b_{2,2Q})f^{\infty}_{2})(x)|\\
&&\lesssim \int_{2Q}\int_{\mathbb{R}^n\backslash 2Q}\frac{|b_{1}(y_{1})-b_{1,2Q}||f_{1}(y_{1})||b_{2}(y_{2})-b_{2,2Q}||f_{1}(y_{1})||f_{2}(y_{2})|}{(|x-y_{1}|^{2}+|x-y_{2}|^{2})^{n-\alpha/2}}dy_{2}dy_{1}\\
&&\lesssim \int_{2Q}|b(y_{1})-b_{2Q}||f_{1}(y_{1})|dy_{1}\sum_{k=1}^{\infty}\frac{1}{|2^{k}Q|^{2-\alpha/n}}\int_{2^{k+1}Q\backslash2^{k}Q}|b_{2}(y_{2})-b_{2,2Q}||f_{2}(y_{2})|dy_{2}\\
&&\lesssim  \bigg(\frac{1}{|2Q|}\int_{2Q}|b(y_{1})-b_{2Q}|^{p'_{1}}dy_{1}\bigg)^{1/p'_{1}}\|b_{2}\|_{\rm BMO}\sum_{k=1}^{\infty}k|2^{k+1}Q|^{\alpha/n-1/p_{0}}\|(f_{1},f_{2})\|_{\mathcal{M}^{p_{0}}_{\vec{P}}}\\
&&\lesssim \|b_{1}\|_{\rm BMO}\|b_{2}\|_{\rm BMO}|Q|^{\alpha/n-1/p_{0}}\|(f_{1},f_{2})\|_{\mathcal{M}^{p_{0}}_{\vec{P}}}.
\end{eqnarray*}

Since the operator $[\Pi\vec{b}, I_{\alpha}]$ can be devided into the following parts:
$$[\Pi\vec{b}, I_{\alpha}](f^{0}_{1},f^{\infty}_{2})(x)={\rm J_{21}+J_{22}+J_{23}+J_{24}},$$
where
\begin{eqnarray*}
&&{\rm J_{21}}:=(b_{1}(x)-b_{1,2Q})(b_{2}-b_{2,2Q})I_{\alpha}(f^{0}_{1},f^{\infty}_{2})(x);\\
&&{\rm J_{22}}:=(b_{1}(x)-b_{1,2Q})I_{\alpha}(f^{0}_{1},(b_{2}-b_{2,2Q})f^{\infty}_{2})(x);\\
&&{\rm J_{23}}:=(b_{2}(x)-b_{2,2Q})I_{\alpha}((b_{1}-b_{1,2Q})f^{0}_{1},f^{\infty}_{2})(x);\\
&&{\rm J_{24}}:=I_{\alpha}((b_{1}-b_{1,2Q})f^{0}_{1},(b_{2}-b_{2,2Q})f^{\infty}_{2})(x).
\end{eqnarray*}
This yields
\begin{eqnarray*}
{\rm J_{2}}&\lesssim&|Q|^{1/q_{0}-1/q}\bigg(\int_{Q}|{\rm J_{21}+J_{22}+J_{23}+J_{24}}|^{q}dx\bigg)^{1/q}\\
&\lesssim& \|b_{1}\|_{\rm BMO}\|b_{2}\|_{\rm BMO}\|(f_{1},f_{2})\|_{\mathcal{M}^{p_{0}}_{\vec{P}}}.
\end{eqnarray*}

Combining all the estimates for terms ${\rm J_{1},J_{2},J_{3},J_{4}}$, we get
$$\|[\Pi\vec{b},I_{\alpha}](f_{1},f_{2})\|_{M^{q_{0}}_{q}}\lesssim \|b_{1}\|_{\rm BMO}\|b_{2}\|_{\rm BMO}\|f_{1},f_{2}\|_{\mathcal{M}^{p_{0}}_{\vec{P}}}.$$

\vskip 0.5cm
\noindent
{\it Proof of (2).} Let $z_{0}\in \mathbb{R}^n$ such that $|(z_{0},z_{0})|>2\sqrt{n}$. Take $\mathbb{B}=B\big((z_{0},z_{0}),\sqrt{2n}\big)\subset \mathbb{R}^{2n}$. Since $O\notin \mathbb{B}$, then we can express $(|y_{1}|^{2}+|y_{2}|^{2})^{n-\alpha/2}$ as an absolutely convergent Fourier series of the form
$$(|y_{1}|^{2}+|y_{2}|^{2})^{n-\alpha/2}=\sum_{j}a_{j}e^{iv_{j}\cdot(y_{1},y_{2})}, \quad (y_{1},y_{2})\in B,$$
where $\sum_{j}|a_{j}|<\infty$ and we do not care about the vectors $v_{j}\in \mathbb{R}^{2n},$ but we will at times express them as $v_{j}=(v_{j}^{1},v_{j}^{2})\in \mathbb{R}^{n}\times \mathbb{R}^n.$

Let $Q=Q(x_{0},r)$ be any arbitrary cube in $\mathbb{R}^n$. Set $\tilde{z}=x_{0}+rz_{0}$ and take $Q'=Q(\tilde{z},r)\subset \mathbb{R}^n$. So for any $x\in Q$ and $y_{1},y_{2}\in Q'$, we have
$$\Big|\frac{x-y_{1}}{r}-z_{0}\Big|\leq \Big|\frac{x-x_{0}}{r}\Big|+\Big|\frac{y_{1}-\tilde{z}}{r}\Big|\leq \sqrt{n},
\quad \Big|\frac{x-y_{2}}{r}-z_{0}\Big|\leq \Big|\frac{x-x_{0}}{r}\Big|+\Big|\frac{y_{2}-\tilde{z}}{r}\Big|\leq \sqrt{n},$$
which implies that
$$\bigg(\Big|\frac{x-y_{1}}{r}-z_{0}\Big|^{2}+\Big|\frac{x-y_{2}}{r}-z_{0}\Big|^{2}\bigg)^{1/2}\leq \sqrt{2n};$$
that is, $(\frac{x-y_{1}}{r},\frac{x-y_{2}}{r})\in \mathbb{B}$.

Let $s(x)=\overline{\mathrm{sgn}(\int_{Q'}(b(x)-b(y))dy)}$. We have the following estimate,
\begin{eqnarray*}
&&\frac{1}{|Q|}\int_{Q}|b(x)-b_{Q'}|^{2}dx\\
&&\lesssim \frac{1}{|Q|}\int_{Q}s^{2}(x)(b(x)-b_{Q'})^{2}dx\\
&&\lesssim \frac{1}{|Q||Q'|^{2}}\int_{Q}\int_{Q'}\int_{Q'}s^{2}(x)\big(b(x)-b(y_{1})\big)\big(b(x)-b(y_{2})\big)dy_{1}dy_{2}dx\\
&&\lesssim \frac{r^{2n-\alpha}}{|Q|^{3}}\int_{Q}\int_{Q'}\int_{Q'}\frac{s^{2}(x)\big(b(x)-b(y_{1})\big(b(x)-b(y_{2})\big)\big)}
{\big(|x-y_{1}|^{2}+|x-y_{2}|^{2}\big)^{n-\alpha/2}}\sum_{j}a_{j}e^{iv_{j}\cdot(\frac{x-y_{1}}{r},\frac{x-y_{2}}{r})}dy_{1}dy_{2}dx.
\end{eqnarray*}
Setting
$$g_{j}(y_{1})=e^{-\frac{i}{r}v^{1}_{j}\cdot y_{1}}\chi_{Q'}(y_{1}),$$
$$h_{j}(y_{2})=e^{-\frac{i}{r}v^{2}_{j}\cdot y_{2}}\chi_{Q'}(y_{2}),$$
$$m_{j}(x)=e^{\frac{i}{r}v_{j}\cdot (x,x)}\chi_{Q}(x)s^{2}(x).$$
We have
\begin{eqnarray*}
&&\frac{1}{|Q|}\int_{Q}|b(x)-b_{Q}|^{2}dx\lesssim\frac{1}{|Q|}\int_{Q}|b(x)-b_{Q'}|^{2}dx\\
&&\lesssim \frac{r^{2n-\alpha}\delta^{-2n+\alpha}}{|Q|^{3}}\sum_{j}|a_{j}|\int_{Q}\big|[\Pi\vec{b},I_{\alpha}](g_{j},h_{j})(x)m_{j}(x)\big|dx\\
&&\lesssim \frac{r^{2n-\alpha}}{|Q|^{2+1/q}}\sum_{j}|a_{j}|\bigg(\int_{Q}\big|[\Pi\vec{b},I_{\alpha}](g_{j},h_{j})(x)\big|^{q}dx\bigg)^{1/q}\\
&&\lesssim r^{-\alpha-n/q}\sum_{j}|a_{j}|\|[\Pi\vec{b},I_{\alpha}](g_{j},h_{j})\|_{M^{q_{0}}_{q}}\\
&&\lesssim \|[\Pi\vec{b},I_{\alpha}]\|_{\mathcal{M}^{p_{0}}_{\vec{P}}\rightarrow M^{q_{0}}_{q}}\sum_{j}|a_{j}|.
\end{eqnarray*}
The desired result follows from here.
\end{proof}

As mentioned in the introduction, ${\rm CMO}$ is the closure in ${\rm BMO}$ of the space of $C^{\infty}$ functions with compact support. In \cite{U}, it was shown that ${\rm CMO}$ can be characterized in the following way.
\begin{lemma}(\cite{U})\label{lem4}
Let $b\in {\rm BMO}$. Then $b$ is in ${\rm CMO}$ if and only if

\begin{equation}\label{2.1.1}
\lim_{a\rightarrow 0}\sup_{|Q|=a}\frac{1}{|Q|}\int_{Q}|b(x)-b_{Q}|dx=0;
\end{equation}

\begin{equation}\label{2.1.2}
\lim_{a\rightarrow \infty}\sup_{|Q|=a}\frac{1}{|Q|}\int_{Q}|b(x)-b_{Q}|dx=0;
\end{equation}

\begin{equation}\label{2.1.3}
\lim_{|y|\rightarrow 0}\frac{1}{|Q|}\int_{Q}|b(x+y)-b_{Q}|dx=0,  \ for\ each ~Q.
\end{equation}
\end{lemma}

\begin{lemma}\label{lem5}
Support that $b\in {\rm BMO}$ with $\|b\|_{*}=1$. If for some $0<\epsilon<1$ and a cube $Q$ with its center at $x_{Q}$ and $r_{Q}$, $b$ is not a constant on cube $Q$ and satisfies
$$\frac{1}{|Q|}\int_{Q}|b(y)-b_{Q}|dy>\epsilon^{1/2},$$
then for the function $f_{i}(i=1,2)$ defined by
\begin{equation}\label{2.3.0}
f_{i}(y_{i})=|Q|^{(\lambda_{i}-1)/(p_{i})}\big(sgn(b(y_{i})-b_{Q})-c_{0}\big)\chi_{Q}(y_{i}),
\end{equation}
where $c_{0}=|Q|^{-1}\int_{Q}sgn\big(b(y)-b_{Q}\big)dy_{i}$ and $\frac{\lambda_{1}}{p_{1}}+\frac{\lambda_{2}}{p_{2}}=\frac{1}{p}-\frac{1}{p_{0}}$ with $0<\lambda_{i}<1$ for $i=1,2$. There exists constants $\gamma_{1},\gamma_{2},\gamma_{3}$ satisfying $\gamma_{2}>\gamma_{1}>2$ and $\gamma_{3}>0$, such that
\begin{equation}\label{2.3.1}
|Q|^{\frac{1}{q_{0}}-\frac{1}{q}}\bigg(\int_{\gamma_{1}r_{Q}<|x-x_{Q}|<\gamma_{2}r_{Q}}\big|[\Pi\vec{b},I_{\alpha}](f_{1},f_{2})(x)\big|^{q}dx\bigg)^{1/q}\geq \gamma_{3},
\end{equation}
\begin{equation}\label{2.3.2}
|Q|^{\frac{1}{q_{0}}-\frac{1}{q}}\bigg(\int_{|x-x_{Q}|>\gamma_{2}r_{Q}}\big|[\Pi\vec{b},I_{\alpha}](f_{1},f_{2})(x)\big|^{q}dx\bigg)^{1/q}\leq \frac{\gamma_{3}}{4}.
\end{equation}

Moreover, there exists a constant $0<\beta<< \gamma_{2}$ depending only on $p_{1},p_{2},n$ such that  for all measurable subsets $E\subset \big\{x:\gamma_{1}r_{Q}<|x-x_{Q}|<\gamma_{2}r_{Q}\big\}$ satisfying $\frac{|E|}{|Q|}<\beta^{n}$, we have
\begin{equation}\label{2.3.3}
|Q|^{\frac{1}{q_{0}}-\frac{1}{q}}\bigg(\int_{E}\big|[\Pi\vec{b},I_{\alpha}](f_{1},f_{2})(x)\big|^{q}dx\bigg)^{1/q}\leq \frac{\gamma_{3}}{4}.
\end{equation}
\end{lemma}

\begin{proof}
Since $\int_{Q}\big(b(y)-b_{Q}\big)dy=0,$ it is easy to check that $f_{i}$ satisfies
$$\mathrm{supp} f_{i}\subset Q,$$
$$f_{i}(y_{i})(b(y)-b_{Q})\geq 0,$$
$$\int f_{i}(y_{i})dy_{i}=0,$$
$$|f_{i}(y_{i})|\leq 2|Q|^{(\lambda_{i}-1)/p_{i}},$$
$$\int \big(b(y)-b_{Q}\big)f_{i}(y)dy=|Q|^{(\lambda_{i}-1)/p_{i}}\int_{Q}|b(y_{i})-b_{Q}|dy.$$
Moreover, it is easy to see that $\|(f_{1},f_{2})\|_{\mathcal{M}^{p_{0}}_{p}}\leq C$. For a cube $Q$ with center $x_{Q}$ and $x\in (2\sqrt{n}Q)^{c}$, the following point-wise estimates hold:
\begin{equation}\label{2.3.4}
|I_{\alpha}\big((b-b_{Q})f_{1},(b-b_{Q})f_{2}\big)(x)|\lesssim |Q|^{2-\frac{1}{p_{0}}}|x-x_{Q}|^{-2n+\alpha},
\end{equation}
\begin{equation}\label{2.3.5}
|I_{\alpha}\big((b-b_{Q})f_{1},(b-b_{Q})f_{2}\big)(x)|\gtrsim \epsilon |Q|^{2-\frac{1}{p_{0}}}|x-x_{Q}|^{-2n+\alpha},
\end{equation}
\begin{equation}\label{2.3.6}
|I_{\alpha}\big((b-b_{Q})f_{1},f_{2}\big)(x)|\lesssim |Q|^{2-\frac{1}{p_{0}}+\frac{1}{n}}|x-x_{Q}|^{-2n+\alpha-1},
\end{equation}
\begin{equation}\label{2.3.7}
|I_{\alpha}\big(f_{1},(b-b_{Q})f_{2}\big)(x)|\lesssim |Q|^{2-\frac{1}{p_{0}}+\frac{1}{n}}|x-x_{Q}|^{-2n+\alpha-1},
\end{equation}
\begin{equation}\label{2.3.8}
|I_{\alpha}\big(f_{1},f_{2}\big)(x)|\lesssim |Q|^{2-\frac{1}{p_{0}}+\frac{1}{n}}|x-x_{Q}|^{-2n+\alpha-1},
\end{equation}
where $f_{i}$ as above and the constants involved are independent of $b, f_{i}$ and $\epsilon$.

To prove (\ref{2.3.4}), from the fact that $\|b\|_{*}=1$ and $x\in (2\sqrt{n}Q)^{c}$, we have
\begin{eqnarray*}
&&|I_{\alpha}\big((b-b_{Q})f_{1},(b-b_{Q})f_{2}\big)(x)|\\
&&=\bigg|\int_{\mathbb{R}^n}\int_{\mathbb{R}^n}\frac{(b(y_{1})-b_{Q})(b(y_{2})-b_{Q})f_{1}(y_{1})f_{2}(y_{2})}
{\big(|x-y_{1}|^{2}+|x-y_{2}|^{2}\big)^{n-\alpha/2}}dy_{1}dy_{2}\bigg|\\
&&\lesssim |x-x_{Q}|^{-2n+\alpha}\prod_{i=1}^{2}
\int_{Q}(b(y_{i})-b_{Q})f_{i}(y_{i})dy_{i}\\
&&\lesssim|Q|^{\frac{\lambda_{1}-1}{p_{1}}+\frac{\lambda_{2}-1}{p_{2}}}|x-x_{Q}|^{-2n+\alpha}\prod_{i=1}^{2}\int_{Q}
|b(y_{i})-b_{Q}|dy_{i}\\
&&\lesssim |Q|^{2+\frac{\lambda_{1}}{p_{1}}+\frac{\lambda_{2}}{p_{2}}-\frac{1}{p_{1}}-\frac{1}{p_{2}}}|x-x_{Q}|^{-2n+\alpha}\\
&&\lesssim|Q|^{2-\frac{1}{p_{0}}}|x-x_{Q}|^{-2n+\alpha}.
\end{eqnarray*}

For (\ref{2.3.5}), using that $\big(b(y_{i})-b_{Q}\big)f_{i}(y_{i})\geq 0$, we can compute
\begin{eqnarray*}
&&|I_{\alpha}\big((b-b_{Q})f_{1},(b-b_{Q})f_{2}\big)(x)|\\
&&=\bigg|\int_{\mathbb{R}^n}\int_{\mathbb{R}^n}\frac{(b(y_{1})-b_{Q})(b(y_{2})-b_{Q})f_{1}(y_{1})f_{2}(y_{2})}
{\big(|x-y_{1}|^{2}+|x-y_{2}|^{2}\big)^{n-\alpha/2}}dy_{1}dy_{2}\bigg|\\
&&\gtrsim |x-x_{Q}|^{-2n+\alpha}\prod_{i=1}^{2}\bigg|\int_{Q}\big(b(y_{i})-b_{Q}\big)f_{i}(y_{i})dy_{i}\bigg|\\
&&=|x-x_{Q}|^{-2n+\alpha}\prod_{i=1}^{2}|Q|^{\frac{\lambda_{i}-1}{p_{i}}}\int_{Q}\big|b(y_{i})-b_{Q}\big|dy_{i}\\
&&\gtrsim \epsilon |Q|^{2-\frac{1}{p_{0}}}|x-x_{Q}|^{-2n+\alpha}.
\end{eqnarray*}

For (\ref{2.3.6}), applying the fact $|f_{2}(y_{2})|\leq 2|Q|^{(\lambda_{2}-1)/p_{2}}$ and $\int_{Q} f_{2}(y_{2})dy_{2}=0$, we can also estimate for any $y'\in Q$,
\begin{eqnarray*}
&&|I_{\alpha}\big((b-b_{Q})f_{1},f_{2}\big)(x)|\\
&&=\bigg|\int_{\mathbb{R}^n}\int_{\mathbb{R}^n}\frac{(b(y_{1})-b_{Q})f_{1}(y_{1})f_{2}(y_{2})}
{\big(|x-y_{1}|^{2}+|x-y_{2}|^{2}\big)^{n-\alpha/2}}dy_{1}dy_{2}\bigg|\\
&&=\bigg|\int_{\mathbb{R}^n}\int_{\mathbb{R}^n}\frac{(b(y_{1})-b_{Q})f_{1}(y_{1})f_{2}(y_{2})}
{\big(|x-y_{1}|^{2}+|x-y_{2}|^{2}\big)^{n-\alpha/2}}dy_{1}dy_{2}\\
&&\qquad -\int_{\mathbb{R}^n}\int_{\mathbb{R}^n}\frac{(b(y_{1})-b_{Q})f_{1}(y_{1})f_{2}(y_{2})}
{\big(|x-y_{1}|^{2}+|x-y'_{2}|^{2}\big)^{n-\alpha/2}}dy_{1}dy_{2}\bigg|\\
&&\lesssim |Q|^{\frac{\lambda_{2}-1}{p_{2}}}\int_{Q}
\int_{Q}\frac{|y_{2}-y'_{2}|(b(y_{1})-b_{Q})f_{1}(y_{1})}{\big(|x-y_{1}|+|x-y_{2}|\big)^{2n-\alpha+1}}dy_{1}dy_{2}\\
&&\lesssim|Q|^{\frac{\lambda_{1}-1}{p_{1}}+\frac{\lambda_{2}}{p_{2}}+\frac{1}{p'_{2}}+\frac{1}{n}}|x-x_{Q}|^{-2n+\alpha-1}\int_{Q}
|b(y_{1})-b_{Q}|dy_{1}\\
&&\lesssim |Q|^{2-\frac{1}{p_{0}}+\frac{1}{n}}|x-x_{Q}|^{-2n+\alpha-1}.
\end{eqnarray*}

It is easy to see that $|I_{\alpha}\big((b-b_{Q})f_{1},f_{2}\big)(x)|=|I_{\alpha}\big(f_{1},(b-b_{Q})f_{2}\big)(x)|$, then (\ref{2.3.7}) holds.

Finally, using that $f_{1}$ has mean zero we obtain (\ref{2.3.8}) as follows.
\begin{eqnarray*}
&&|I_{\alpha}\big(f_{1},f_{2}\big)(x)|\\
&&=\bigg|\int_{\mathbb{R}^n}\int_{\mathbb{R}^n}\frac{f_{1}(y_{1})f_{2}(y_{2})}
{\big(|x-y_{1}|^{2}+|x-y_{2}|^{2}\big)^{n-\alpha/2}}-\frac{f_{1}(y_{1})f_{2}(y_{2})}{\big(|x-y'_{1}|^{2}+|x-y_{2}|^{2}\big)^{n-\alpha/2}}dy_{1}dy_{2}\bigg|\\
&&\lesssim \int_{Q}
\int_{Q}\frac{|y_{1}-y'_{1}||f_{1}(y_{1})||f_{2}(y_{2})|}{\big(|x-y_{1}|+|x-y_{2}|\big)^{2n-\alpha+1}}dy_{1}dy_{2}\\
&&\lesssim  |Q|^{2-\frac{1}{p_{0}}+\frac{1}{n}}|x-x_{Q}|^{-2n+\alpha-1}.
\end{eqnarray*}

Now, we give the proofs of (\ref{2.3.1})-(\ref{2.3.3}). Taking $\nu>16,$ by (\ref{2.3.6}) we obtain
\begin{eqnarray*}
&&\bigg(\int_{|x-x_{Q}|>\nu r_{Q}}\big|(b(x)-b_{Q})I_{\alpha}((b-b_{Q})f_{1},f_{2})(x)\big|^{q}dx\bigg)^{\frac{1}{q}}\\
&&\lesssim |Q|^{2-\frac{1}{p_{0}}+\frac{1}{n}}\sum_{s=\lfloor\log_{2}\nu\rfloor}^{\infty}
\bigg(\int_{2^{s}r_{Q}<|x-x_{Q}|<2^{s+1}r_{Q}}\frac{|b(x)-b_{Q}|^{q}}{|x-x_{Q}|^{q(2n-\alpha+1)}}dx\bigg)^{\frac{1}{q}}\\
&&\lesssim |Q|^{1/q_{0}-1/q}\sum_{s=\lfloor\log_{2}\nu\rfloor}^{\infty}2^{-s(2n-\alpha+1)}
\bigg(\frac{1}{|2^{s+1}Q|}\int_{2^{s}r_{Q}<|x-x_{Q}|<2^{s+1}r_{Q}}|b(x)-b_{Q}|^{q}dx\bigg)^{\frac{1}{q}}\\
&&\lesssim |Q|^{1/q-1/q_{0}}\sum_{s=\lfloor\log_{2}\nu\rfloor}^{\infty}s2^{-s(2n-\alpha+1-\frac{n}{q})}\\
&&\lesssim |Q|^{1/q-1/q_{0}}\sum_{s=\lfloor\log_{2}\nu\rfloor}^{\infty}2^{-s(2n-\alpha-\frac{n}{q}+\frac{1}{2})}\\
&&\lesssim |Q|^{1/q-1/q_{0}}\nu^{-(2n-\alpha-\frac{n}{q}+\frac{1}{2})},
\end{eqnarray*}
where we have used that $s\leq 2^{s/2}$ for $4\leq \lfloor\log_{2}\nu\rfloor\leq s$.

Similarly, we also have
$$\bigg(\int_{|x-x_{Q}|>\nu r_{Q}}\big|(b(x)-b_{Q})I_{\alpha}(f_{1},(b-b_{Q})f_{2})(x)\big|^{q}dx\bigg)^{\frac{1}{q}}\lesssim |Q|^{1/q_{0}-1/q}\nu^{-2n-\alpha+\frac{1}{2}-\frac{n}{q}},$$
$$\bigg(\int_{|x-x_{Q}|>\nu r_{Q}}\big|(b(x)-b_{Q})^{2}I_{\alpha}(f_{1},f_{2})(x)\big|^{q}dx\bigg)^{\frac{1}{q}}\lesssim |Q|^{1/q_{0}-1/q}\nu^{-2n-\alpha+\frac{1}{2}-\frac{n}{q}}.$$

Then for $\mu>\nu,$ using (\ref{2.3.4}), (\ref{2.3.5}) and the estimates above, we get
\begin{eqnarray*}
&&|Q|^{1/q_{0}-1/q}\int_{\nu r_{Q}<|x-x_{Q}|<\mu r_{Q}}\big|[\Pi\vec{b},I_{\alpha}](f_{1},f_{2})(x)\big|^{q}dx\\
&&\geq C|Q|^{1/q_{0}-1/q}\bigg(\int_{\nu r_{Q}<|x-x_{Q}|<\mu r_{Q}}\big|I_{\alpha}((b-b_{Q})f_{1},(b-b_{Q})f_{2})(x)\big|^{q}dx\bigg)^{1/q}\\
&&\qquad -C|Q|^{1/q_{0}-1/q}\bigg(\int_{\nu r_{Q}<|x-x_{Q}|}\big|(b(x)-b_{Q})I_{\alpha}((b-b_{Q})f_{1},f_{2})(x)\big|^{q}dx\bigg)^{1/q}\\
&&\qquad -C|Q|^{1/q_{0}-1/q}\bigg(\int_{\nu r_{Q}<|x-x_{Q}|}\big|(b(x)-b_{Q})I_{\alpha}(f_{1},(b-b_{Q})f_{2})(x)\big|^{q}dx\bigg)^{1/q}\\
&&\qquad -C|Q|^{1/q_{0}-1/q}\bigg(\int_{\nu r_{Q}<|x-x_{Q}|}\big|(b(x)-b_{Q})^{2}I_{\alpha}(f_{1},f_{2})(x)\big|^{q}dx\bigg)^{1/q}\\
&&\geq C\epsilon|Q|^{2-\frac{\alpha}{n}-\frac{1}{q}}\bigg(\int_{\nu r_{Q}<|x-x_{Q}|<\mu r_{Q}}|x-x_{Q}|^{q(-2n+\alpha)}dx\bigg)^{1/q}-C\nu^{-2n+\alpha+\frac{n}{q}-\frac{1}{2}}\\
&&\geq C\epsilon \Big(\nu^{-2nq+n+\alpha q}-\mu^{-2nq+n+\alpha q}\Big)^{\frac{1}{q}}-C\nu^{-2n+\alpha+\frac{n}{q}-\frac{1}{2}}.
\end{eqnarray*}
We can select $\gamma_{1},\gamma_{2}$ in place of $\nu,\mu$ with $\gamma_{2}>>\gamma_{1}$, then (\ref{2.3.1}) and (\ref{2.3.2}) are verified for some $\gamma_{3}>0$.

We now verified (\ref{2.3.3}). Let $E\subset \big\{\gamma_{1}r_{Q}<|x-x_{Q}|<\gamma_{2}r_{Q}\big\}$ be an arbitrary measurable set. It follows from Minkowski inequality that
\begin{eqnarray*}
&&|Q|^{1/q_{0}-1/q}\bigg(\int_{E}\big|[\Pi\vec{b},I_{\alpha}](f_{1},f_{2})(x)\big|^{q}dx\bigg)^{\frac{1}{q}}\\
&&\lesssim |Q|^{2-\alpha/n-1/q}\bigg(\int_{E}|x-x_{Q}|^{-q(2n-\alpha)}dx\bigg)^{\frac{1}{q}}
+|Q|^{2-\alpha/n-1/q+\frac{1}{n}}\bigg(\int_{E}\frac{|b(x)-b_{Q}|^{q}}{|x-x_{Q}|^{q(2n-\alpha+1)}}dx\bigg)^{\frac{1}{q}}\\
&&\qquad +|Q|^{2-\alpha/n-1/q+\frac{1}{n}}\bigg(\int_{E}\frac{|b(x)-b_{Q}|^{2q}}{|x-x_{Q}|^{q(2n-\alpha+1)}}dx\bigg)^{\frac{1}{q}}\\
&&\lesssim \bigg(\frac{|E|^{1/q}}{|Q|^{1/q}}+\Big(\frac{1}{|Q|}\int_{E}|b(x)-b_{Q}|^{q}dx\Big)^{\frac{1}{q}}
+\Big(\frac{1}{|Q|}\int_{E}|b(x)-b_{Q}|^{2q}dx\Big)^{\frac{1}{q}}\bigg)\\
&&\lesssim \frac{|E|^{\frac{1}{2q}}}{|Q|^{\frac{1}{2q}}}\bigg(1+\log\Big(\frac{\tilde{C}|Q|}{|E|}\Big)\bigg)^{\frac{1+\lfloor 2q\rfloor}{2q}}.
\end{eqnarray*}
The last inequality can be obtained by \cite[P.309]{CDW1} taking $0<\beta<\min\{\tilde{C}^{1/n},\gamma_{2}\}$ and sufficiently small so that (\ref{2.3.3}) holds.
\end{proof}

In order to prove Theorem \ref{main1} and \ref{main2}, we need the characterization that a subset of $M^{p_{0}}_{p}$ is a strong pre-compact set.
\begin{lemma}\label{lem6}[\cite{CDW3}]
Let $1<p\leq p_{0}<\infty$. Suppose that the subset $\mathfrak{F}\subset M^{p_{0}}_{p}$ satisfies the following conditions:

(i) norm boundedness uniformly
$$
\sup_{f\in \mathfrak{F}}\|f\|_{M^{p_{0}}_{p}}<\infty;
$$

(ii) control uniformly away from the origin
$$
\lim_{A\rightarrow \infty}\|f\chi_{E_{A}}\|_{M^{p_{0}}_{p}}=0 \ uniformly\ in\ f\in \mathfrak{F}, \text{where}~E_{A}=\{x\in \mathbb{R}^n:|x|>A\};
$$
(iii) translation continuity uniformly
$$
\lim_{y\rightarrow 0}\|f(\cdot+y)-f(\cdot)\|_{M^{p_{0}}_{p}}=0  \ uniformly\ in\ f\in \mathfrak{F};
$$

then $\mathfrak{F}$ is pre-compact in $M^{p_{0}}_{p}$.
\end{lemma}

\section{Proof of Theorem \ref{main1} and Theorem \ref{main2}}

\vskip 0.5cm
\noindent
{\it Proof of Theorem \ref{main1}.}
We need only to show the set $\mathfrak{F}:=\{[\Pi\vec{b},T](f_{1},f_{2}): \|(f_{1},f_{2})\|_{\mathcal{M}^{p_{0}}_{\vec{P}}}\leq 1\}$ is strong pre-compact in $M^{p_{0}}_{p}$ when $b_{1},b_{2}\in C^{\infty}_{c}$. By Lemma \ref{lem6}, we need to verify the conditions $(i), (ii)$ and $(iii)$ hold uniformly in $\mathfrak{F}$ for $b_{1},b_{2}\in C^{\infty}_{c}$.

It is easy to verify that $\mathfrak{F}$ satisfies the condition $(i)$ by Lemma \ref{lem1}.

As for condition $(ii)$, suppose that ${\rm supp}~ b_{1}, b_{2}\subset \{x\in \mathbb{R}^{n}:|x|\leq \beta\}$ with $\beta>1$. For any $A>2\beta$,
$$
|Q|^{1/p_{0}-1/p}\bigg(\int_{Q}|[\Pi\vec{b},T](f_{1},f_{2})(x)|^{p}\chi_{E_{A}}(x)dx\bigg)^{1/p}\rightarrow 0, (A\rightarrow +\infty).
$$
In fact, for any $x\in E_{A}$, $|x|\leq |x-y_{1}|+|y_{1}|\leq |x-y_{1}|+|x-y_{2}|$, then
\begin{eqnarray*}
|[\Pi\vec{b},T](f_{1},f_{2})(x)|&\lesssim &\|b_{1}\|_{\infty}\|b_{2}\|_{\infty}\int_{|y_{1}|\leq \beta}\int_{|y_{2}|\leq \beta}\frac{|f_{1}(y_{1})f_{2}(y_{2})|}{(|x-y_{1}|+|x-y_{2}|)^{2n}}dy_{1}dy_{2}\\
&\lesssim& \frac{1}{|x|^{2n}}\int_{|y_{1}|\leq \beta}\int_{|y_{2}|\leq \beta}|f_{1}(y_{1})f_{2}(y_{2})|dy_{1}dy_{2}\\
&\lesssim& \beta^{2n-n/p_{0}}|x|^{-2n}\|(f_{1},f_{2})\|_{\mathcal{M}^{p_{0}}_{\vec{P}}}.
\end{eqnarray*}
From $L^{p_{0}}\subset M^{p_{0}}_{q}$, it follows that for any cube $Q$
\begin{eqnarray*}
&&|Q|^{1/p_{0}-1/p}\bigg(\int_{Q}|[\Pi\vec{b},T](f_{1},f_{2})(x)|^{p}\chi_{E_{A}}(x)dx\bigg)^{1/p}\\
&&\lesssim \bigg(\int_{|x|>A}\beta^{2n-n/p_{0}}|x|^{-2n}dx\bigg)^{1/p_{0}}\\
&&\lesssim \Big|\frac{\beta}{A}\Big|^{2n-n/p_{0}}.
\end{eqnarray*}
Thus, the inequality above tends to zero as $A\rightarrow \infty$.

Finally, it remains to prove condition $(iii)$. We need to show that for any $0<\epsilon<1$, if $|t|$ is sufficiently small depending $\epsilon$, then
$$\|[\Pi\vec{b},T](f_{1},f_{2})(\cdot+t)-[\Pi\vec{b},T](f_{1},f_{2})(\cdot)\|_{M^{p_{0}}_{p}}\lesssim \epsilon.$$

To do this, we break
$$[\Pi\vec{b}, T](f_{1},f_{2})(x)-[\Pi\vec{b}, T](f_{1},f{2})(x+t)$$
into a sum of four terms
$${\rm II_{1}+II_{2}+II_{3}+II_{4}}$$
with
\begin{eqnarray*}
&&{\rm II_{1}}:=\iint_{\Omega}K(x,y_{1},y_{2})(b_{1}(x+t)-b_{1}(x))(b_{2}(y_{2})-b_{2}(x))f_{1}(y_{1})f_{2}(y_{2})dy_{1}dy_{2};\\
&&{\rm II_{2}}:=\iint_{\Omega}\Big(K(x,y_{1},y_{2})(b_{2}(y_{2})-b_{2}(x))-K(x+t,y_{1},y_{2})(b_{2}(y_{2})-b_{2}(x+t))\Big)\\
&&\qquad \qquad\times(b_{1}(y_{1})-b_{1}(x+t))f_{1}(y_{1})f_{2}(y_{2})dy_{1}dy_{2};\\
&&{\rm II_{3}}:=\iint_{\Omega^{c}}K(x,y_{1},y_{2})(b_{1}(y_{1})-b_{1}(x))(b_{2}(y_{2})-b_{2}(x))f_{1}(y_{1})f_{2}(y_{2})dy_{1}dy_{2};\\
&&{\rm II_{4}}:=\iint_{\Omega^{c}}K(x+t,y_{1},y_{2})(b_{1}(y_{1})-b_{1}(x+t))(b_{2}(x+t)-b_{2}(y_{2}))f_{1}(y_{1})f_{2}(y_{2})dy_{1}dy_{2},
\end{eqnarray*}
where $\Omega=\{(y_{1},y_{2}):|x-y_{1}|+|x-y_{2}|>\delta:=\epsilon^{-1} |t|\}$ and $\Omega^{c}=\mathbb{R}^{2n}\backslash \Omega$.

For ${\rm II_{1}}$, we can compute
\begin{eqnarray*}
{\rm II_{1}(x,t)}&\lesssim&|t|\|\nabla b_{1}\|_{\infty}\big(T_{*}(f_{1},b_{2}f_{2})(x)+|b_{2}(x)|T_{*}(f_{1},f_{2})(x)\big).
\end{eqnarray*}
By Corollary \ref{c} and $b_{2}\in L^{\infty}$, we obtain
$$\|{\rm II_{1}}\|_{M^{p_{0}}_{p}}\lesssim
|t|\bigg(\|(f_{1},b_{2}f_{2})\|_{\mathcal{M}^{p_{0}}_{\vec{P}}}+\|b_{2}\|_{\infty}\|(f_{1},f_{2})\|_{\mathcal{M}^{p_{0}}_{\vec{P}}}\bigg)\lesssim |t|.$$

To deal with the ${\rm II_{2}}$ term, we write ${\rm II_{2}}$ as a sum of three terms, ${\rm II_{21}}+{\rm II_{22}}+{\rm II_{23}}$, where
\begin{eqnarray*}
&&{\rm II_{21}}=\iint_{\Omega}\Big(K(x,y_{1},y_{2})-K(x+t,y_{1},y_{2})\Big)(b_{1}(y_{1})-b_{1}(x+t))f_{1}(y_{1})b_{2}(y_{2})f_{2}(y_{2})dy_{1}dy_{2};\\
&&{\rm II_{22}}=b_{2}(x+t)\iint_{\Omega}\Big(K(x,y_{1},y_{2})-K(x+t,y_{1},y_{2})\Big)(b_{1}(y_{1})-b_{1}(x))f_{1}(y_{1})f_{2}(y_{2})dy_{1}dy_{2}.\\
&&{\rm II_{23}}=\iint_{\Omega}K(x,y_{1},y_{2})(b_{1}(y_{1})-b_{1}(x+t))(b_{2}(x+t)-b_{2}(x))f_{1}(y_{1})f_{2}(y_{2})dy_{1}dy_{2}.
\end{eqnarray*}

By the regularity condition of function $K$, we have
\begin{eqnarray*}
{\rm II_{21}}&\lesssim& \|b_{1}\|_{\infty}\|b_{2}\|_{\infty}\iint_{\Omega}\frac{|t|^{\gamma}|f_{1}(y_{1})||f_{2}(y_{2})|}{(|x-y_{1}|+|x-y_{2}|)^{2n+\gamma}}dy_{1}dy_{2}\\
&\lesssim&|t|^{\gamma}\sum_{k=1}^{\infty}\iint_{2^{k-1}\delta\leq |x-y_{1}|+|x-y_{2}|<2^{k}\delta}\frac{|f_{1}(y_{1})||f_{2}(y_{2})|}{(|x-y_{1}|+|x-y_{2}|)^{2n+\gamma}}dy_{1}dy_{2}\\
&\lesssim& |t|^{\gamma}\delta^{-\gamma}\mathcal{M}(f_{1},f_{2})(x),
\end{eqnarray*}
where the bilinear maximal operator $\mathcal{M}$ is defined by Lerner et al.\cite{LOPTT}, which is used to obtain a precise control on multinear singular integral operators. The bilinear maximal operator $\mathcal{M}$ is defined by
$$\mathcal{M}(f_{1},f_{2})(x)=\sup_{Q\ni x}\prod_{i=1}^{2}\frac{1}{|Q|}\int_{Q}|f_{i}(y_{i})|dy_{i}.$$
By the boundedness of $\mathcal{M}$ from $\mathcal{M}^{p_{0}}_{\vec{P}}$ to $M^{p_{0}}_{p}$ (see \cite{ISST}), we obtain
$$\|{\rm II_{21}}\|_{M^{p_{0}}_{p}}\lesssim|t|^{\gamma}\delta^{-\gamma}
\|\mathcal{M}(f_{1},f_{2})\|_{M^{p_{0}}_{p}}\lesssim |t|^{\gamma}\delta^{-\gamma}.$$

The term ${\rm II_{22}}$ is estimated using the same methods as in the proof for ${\rm II_{21}}$, and the term ${\rm II_{23}}$ is same as ${\rm II_{1}}$. Then
$$\|{\rm II_{2}}\|_{M^{p_{0}}_{p}}\lesssim |t|+|t|^{\gamma}\delta^{-\gamma}.$$

Note that
\begin{eqnarray*}
{\rm II_{3}}&\lesssim& \|\nabla b_{1}\|_{L^{\infty}} \|\nabla b_{2}\|_{L^{\infty}}\iint_{\Omega^{c}}\frac{|f_{1}(y_{1})||f_{2}(y_{2})|}{(|x-y_{1}|+|x-y_{2}|)^{2n-2}}dy_{1}dy_{2}\\
&\lesssim&\sum_{k=1}^{\infty}\iint_{2^{-k}\delta<|x-y_{1}|+|x-y_{2}|\leq 2^{-k+1}\delta}\frac{|f_{1}(y_{1})||f_{2}(y_{2})|}{(|x-y_{1}|+|x-y_{2}|)^{2n-2}}dy_{1}dy_{2}\\
&\lesssim&\delta^{2}\mathcal{M}(f_{1},f_{2})(x),
\end{eqnarray*}
which gives that
$$\|{\rm II_{3}}\|_{M^{p_{0}}_{p}}\lesssim \delta^{2}.$$
Finally, for the last term we proceed in an analogous manner, by replacing $x$ with $x+t$ and the region of integration $\{(y_{1},y_{2}):|x-y_{1}|+|x-y_{2}|<\delta\}$ with the larger one $\{(y_{1},y_{2}):|x+t-y_{1}|+|x+t-y_{2}|<\delta+2|t|\}$. Thus,
$$\|{\rm II_{4}}\|_{M^{p_{0}}_{p}}\lesssim (\delta+2|t|)^{2}.$$

For any $0<\epsilon<1$, there exists a constant $t_{0}=\epsilon^{2}$, for any $|t|<t_{0}$,
$$\|[\Pi\vec{b},T](f_{1},f_{2})(\cdot+t)-[\Pi\vec{b},T](f_{1},f_{2})(\cdot)\|_{_{M^{p_{0}}_{p}}}\lesssim |t|+\epsilon^{\gamma}+\frac{|t|}{\epsilon}+(\frac{1}{\epsilon}+2)^{2}|t|^{2}\lesssim \epsilon^{\gamma}.$$
We prove that condition $(iii)$ holds for $[\Pi\vec{b},T](f_{1},f_{2})$ uniformly in $\mathfrak{F}$ and Theorem \ref{main1} follows.

\vskip 0.5cm
\noindent
{\it Proof of Theorem \ref{main2}.}
We need only to verify the conditions $(i), (ii)$ and $(iii)$ hold uniformly in $\mathcal{G}$ for $b_{1},b_{2}\in C^{\infty}_{c}$,
where
$$\mathcal{G}=\big\{[\Pi\vec{b},I_{\alpha}](f_{1},f_{2}): \|(f_{1},f_{2})\|_{\mathcal{M}^{p_{0}}_{\vec{P}}}\leq 1\big\}.$$

By Lemma \ref{lem3}, we have $\mathcal{G}$ is uniformly bounded.

For the condition $(ii)$, suppose that ${\rm supp} ~b_{1}, b_{2}\subset \{x\in \mathbb{R}^{n}:|x|\leq \beta\}$ with $\beta>1$ and let $A\geq 2\beta$. Then for any $|x|>A$ and $|y_{1}|,|y_{2}|\leq \beta$, we have $|x-y_{1}|^{2}+|x-y_{2}|^{2}\gtrsim |x|^{2}$. Thus, for any cube $Q$,
\begin{eqnarray*}
&&|Q|^{1/q_{0}-1/q}\bigg(\int_{Q}|[\Pi\vec{b},I_{\alpha}](f_{1},f_{2})(x)|^{q}\chi_{E_{A}}(x)dx\bigg)^{1/q}\\
&&\lesssim |Q|^{1/q_{0}-1/q}\bigg(\int_{Q}\bigg(\int_{|y_{1}|\lesssim \beta}\int_{|y_{2}|\lesssim \beta}\frac{|f_{1}(y_{1})f_{2}(y_{2})|dy_{1}dy_{2}}{(|x-y_{1}|^{2}+|x-y_{2}|^{2})^{n-\alpha/2}}\bigg)^{q}\chi_{E_{A}}(x)dx\bigg)^{1/q}\\
&&\lesssim|Q|^{1/q_{0}-1/q}\bigg(\int_{Q}|x|^{\alpha q-2nq}\beta^{2nq-nq/q_{0}}\chi_{E_{A}}(x)dx\bigg)^{1/q}\|(f_{1},f_{2})\|_{\mathcal{M}^{p_{0}}_{\vec{P}}}\\
&&\lesssim \beta^{2n-n/q_{0}}\bigg(\int_{|x|>A}|x|^{\alpha q_{0}-2nq_{0}}dx\bigg)^{1/q_{0}}\\
&&\lesssim \big(\frac{\beta}{A}\big)^{2n-n/q_{0}}.
\end{eqnarray*}
Thus, (b) holds by letting $A\rightarrow \infty$.

To prove the uniform continuity of $\mathcal{G}$, we must see that
$$\lim_{t\rightarrow 0}\|[\Pi\vec{b},I_{\alpha}](f_{1},f_{2})(\cdot+t)-[\Pi\vec{b},I_{\alpha}](f_{1},f_{2})(\cdot)\|_{M^{q_{0}}_{q}}=0.$$

To deal with compactness of fractional integral operators, we find it convenient to use smooth truncations of $I_{\alpha}$. The operator $I^{\delta}_{\alpha}$ is defined by a smooth kernel $K^{\delta}(x,y_{1},y_{2})$ such that
$$K^{\delta}(x,y_{1},y_{2})=\frac{1}{(|x-y_{1}|^{2}+|x-y_{2}|^{2})^{n-\alpha/2}}$$
for $|x-y_{1}|+|x-y_{2}|>2\delta$;
$$K^{\delta}(x,y_{1},y_{2})=0$$
for $|x-y_{1}|+|x-y_{2}|\leq \delta$;
and
$$|\partial^{\gamma} K^{\delta}(x,y_{1},y_{2})|\lesssim \frac{1}{(|x-y_{1}|+|x-y_{2}|)^{2n-\alpha+\gamma}}$$
for all $(x,y_{1},y_{2})$ and all-multi-indexes with $|\gamma|\leq 1$.

Then, we need only to show that
\begin{equation}\label{1}
\lim_{t\rightarrow 0}\|[\Pi\vec{b},I^{\delta}_{\alpha}](f_{1},f_{2})(\cdot+t)-[\Pi\vec{b},I^{\delta}_{\alpha}](f_{1},f_{2})(\cdot)\|_{M^{q_{0}}_{q}}=0,
\end{equation}
where $\delta=|t|^{1/2}$. In fact, for any $x\in \mathbb{R}^n$
\begin{eqnarray*}
&&\Big|[\Pi\vec{b},I_{\alpha}](f_{1},f_{2})(x)-[\Pi\vec{b},I^{\delta}_{\alpha}](f_{1},f_{2})(x)\Big|\\
&&\leq \iint_{|x-y_{1}|+|x-y_{2}|\leq 2\delta}
\frac{|b_{1}(x)-b_{1}(y_{1})||b_{2}(x)-b_{2}(y_{2})||f_{1}(y_{1})||f_{2}(y_{2})|}{(|x-y_{1}|^{2}+|x-y_{2}|^{2})^{n-\alpha/2}}dy_{1}dy_{2}\\
&&\lesssim \|\nabla b_{1}\|_{\infty} \|\nabla b_{2}\|_{\infty}\iint_{|x-y_{1}|+|x-y_{2}|\leq 2\delta}
\frac{|f_{1}(y_{1})||f_{2}(y_{2})|}{(|x-y_{1}|^{2}+|x-y_{2}|^{2})^{n-\alpha/2-1}}dy_{1}dy_{2}\\
&&\lesssim\sum_{k=0}^{\infty}(2^{-k}\delta)^{\alpha+2-2n}\iint_{2^{-k+1}\delta< |x-y_{1}|+|x-y_{2}|\leq 2^{k}\delta}|f_{1}(y_{1})||f_{2}(y_{2})|dy_{1}dy_{2}\\
&&\lesssim\sum_{k=0}^{\infty}(2^{-k}\delta)^{2}\mathcal{M}_{\alpha}(f_{1},f_{2})(x)\\
&&\lesssim \delta^{2}I_{\alpha}(|f_{1}|,|f_{2}|)(x).
\end{eqnarray*}
Set $\delta^{2}=|t|$, which gives that
$$\|[\Pi\vec{b},I_{\alpha}](f_{1},f_{2})-[\Pi\vec{b},I^{\delta}_{\alpha}](f_{1},f_{2})\|_{M^{q_{0}}_{q}}\lesssim |t|.$$

To prove (\ref{1}), we write
\begin{eqnarray*}
&&[\Pi\vec{b},I_{\alpha}^{\delta}](f_{1},f_{2})(x)-[\Pi\vec{b},I_{\alpha}^{\delta}](f_{1},f_{2})(x+t)\\
&&=\int_{\mathbb{R}^n}\int_{\mathbb{R}^n}\big(b_{1}(x)-b_{1}(y_{1})\big)\big(b_{2}(x)-b(y_{2})\big)K^{\delta}(x,y_{1},y_{2})f_{1}(y_{1})f_{2}(y_{2})dy_{1}dy_{2}\\
&&\quad-\int_{\mathbb{R}^n}\int_{\mathbb{R}^n}\big(b_{1}(x+t)-b(y_{1})\big)\big(b_{2}(x+t)-b(y_{2})\big)K^{\delta}(x+t,y_{1},y_{2})f_{1}(y_{1})f_{2}(y_{2})dy_{1}dy_{2}\\
&&=\big(b_{1}(x)-b_{1}(x+t)\big)\big(b_{2}(x)-b_{2}(x+t)\big)\int_{\mathbb{R}^n}\int_{\mathbb{R}^n}K^{\delta}(x,y_{1},y_{2})f_{1}(y_{1})f_{2}(y_{2})dy_{1}dy_{2}\\
&&\quad +\big(b_{1}(x)-b_{1}(x+t)\big)\int_{\mathbb{R}^n}\int_{\mathbb{R}^n}\big(b_{2}(x+t)-b(y_{2})\big)K^{\delta}(x,y_{1},y_{2})f_{1}(y_{1})f_{2}(y_{2})dy_{1}dy_{2}\\
&&\quad +\big(b_{2}(x)-b_{2}(x+t)\big)\int_{\mathbb{R}^n}\int_{\mathbb{R}^n}\big(b_{1}(x+t)-b(y_{1})\big)K^{\delta}(x,y_{1},y_{2})f_{1}(y_{1})f_{2}(y_{2})dy_{1}dy_{2}\\
&&\quad +\int_{\mathbb{R}^n}\int_{\mathbb{R}^n}\big(b_{1}(x+t)-b(y_{1})\big)\big(b_{2}(x+t)-b(y_{2})\big)\\
&&\qquad \qquad\times \big(K^{\delta}(x,y_{1},y_{2})-K^{\delta}(x+t,y_{1},y_{2})\big)f_{1}(y_{1})f_{2}(y_{2})dy_{1}dy_{2}\\
&&={\rm JJ_{1}+JJ_{2}+JJ_{3}+JJ_{4}}.
\end{eqnarray*}
For ${\rm JJ_{1}}$, we simply have
$$|{\rm JJ_{1}}|\lesssim |t|^{2}\|\nabla b_{1}\|_{\infty}\|\nabla b_{2}\|_{\infty}I_{\alpha}(|f_{1}|,|f_{2}|)(x),$$
which implies that
$$\|{\rm JJ_{1}}\|_{M^{q_{0}}_{q}}\lesssim |t|^{2}.$$

Similarly, we also have that for $j=2,3$
$$|{\rm JJ_{j}}|\lesssim |t|\|\nabla b_{1}\|_{\infty}\|b_{2}\|_{\infty}I_{\alpha}(|f_{1}|,|f_{2}|)(x).$$

We now give the estimate for ${\rm JJ_{4}}$. We may assume that $t$ small enough such that $|t|\in (0,\frac{\delta}{4})$. If $|x-y_{1}|+|x-y_{2}|\leq \delta/2$ we have $|x+t-y_{1}|+|x+t-y_{2}|\leq \delta/2+2|t|\leq \delta$, thus
$$K^{\delta}(x+t,y_{1},y_{2})-K^{\delta}(x,y_{1},y_{2})=0.$$
This gives us that
\begin{eqnarray*}
\big|{\rm JJ_{4}}\big|&\lesssim& |t|\|b_{1}\|_{\infty}\|b_{2}\|_{\infty}\iint_{|x-y_{1}|+|x-y_{2}|> \frac{\delta}{2}}\frac{|f_{1}(y_{1})f_{2}(y_{2})|}{\big(|x-y_{1}|+|x-y_{2}|\big)^{2n-\alpha+1}}dy_{1}dy_{2}\\
&\lesssim& |t|\sum_{j=0}^{\infty}\iint_{2^{j-1}\delta<|x-y_{1}|+|x-y_{2}|< 2^{j}\delta}\frac{|f_{1}(y_{1})f_{2}(y_{2})|}{\big(|x-y_{1}|+|x-y_{2}|\big)^{2n-\alpha+1}}dy_{1}dy_{2}\\
&\lesssim& |t|\sum^{\infty}_{j=0}(2^{j}\delta)^{-1}\iint_{2^{j-1}\delta<|x-y_{1}|+|x-y_{2}|< 2^{j}\delta}\frac{|f_{1}(y_{1})f_{2}(y_{2})|}{\big(|x-y_{1}|+|x-y_{2}|\big)^{2n-\alpha}}dy_{1}dy_{2}\\
&\lesssim&\frac{|t|}{\delta}I_{\alpha}(|f_{1}|,|f_{2}|)(x).
\end{eqnarray*}

Combining the estimates above and $t$ small enough such that $0<4t<\delta:=|t|^{1/2}$, we have
$$\|[\Pi\vec{b},I_{\alpha}^{\delta}](f_{1},f_{2})(\cdot+t)-[\Pi\vec{b},I_{\alpha}^{\delta}](f_{1},f_{2})(\cdot)\|_{_{M^{q_{0}}_{q}}}\lesssim (|t|^{2}+|t|+\frac{|t|}{\delta})\|I_{\alpha}(|f_{1}|,|f_{2}|)\|_{M^{q_{0}}_{q}}\lesssim |t|^{1/2}.$$

Thus, the compactness of $[\Pi \vec{b},I_{\alpha}]$ on Morrey space is completed.
\vspace{0.3cm}

So it remains to show that if $b_{1}=b_{2}$ and $[\Pi\vec{b},I_{\alpha}]$ is a compact operator from $\mathcal{M}^{p_{0}}_{\vec{P}}$ to $M^{q_{0}}_{q}$, then $b_{1},b_{2}\in {\rm CMO}$.

First, Lemma \ref{lem3} implies that $b\in {\rm BMO}$. To prove $b$ be an element of ${\rm CMO}$, we will adapt some arguments from \cite{CDW1}, see also \cite{CT}, which in turn are based on the original work in \cite{U}. The approach is the following: if one of the conditions Eqs.(\ref{2.1.1})-(\ref{2.1.3}) in Lemma \ref{lem4} is failed, we will show that there exist sequences of functions, $\{f_{j}\}_{j}$ and $\{g_{j}\}_{j}$ uniformly bounded on $\mathcal{M}^{p_{0}}_{\vec{P}}$, such that $[[\Pi\vec{b},I_{\alpha}](f_{j},g_{j})$ has no convergent subsequence, which contradicts the assumption that $[\Pi\vec{b},I_{\alpha}]$ is compact. It gives us that if $[\Pi\vec{b},I_{\alpha}]$ is compact, $b$ must satisfy all three conditions; that is $b\in {\rm CMO}$.

By Lemma \ref{lem5}, it is sufficient to once again repeat the steps preformed in \cite{CT} (or\cite{CDW1},\cite{CDW2}) to obtain the desired result and it is left to the reader. \qed

{\bf Acknowledgments} We would like to thank the anonymous referee for his/her comments.

\color{black}

\end{document}